\documentclass[12pt]{amsart}

\usepackage{latexsym, amsmath, amssymb, amsfonts, amscd, amsthm, mathrsfs}
\usepackage[all, cmtip]{xy}

\newcommand{\Cb}{{\mathbb C}}
\newcommand{\Rb}{{\mathbb R}}

\newcommand{\Zb}{{\mathbb Z}}
\newcommand{\Nb}{{\mathbb N}}
\newcommand{\Tb}{{\mathbb T}}

\newcommand{\cA}{{\mathcal A}}

\newcommand{\cE}{{\mathcal E}}

\newcommand{\fg}{{\mathfrak g}}
\newcommand{\fk}{{\mathfrak k}}
\newcommand{\fp}{{\mathfrak p}}

\newcommand{\GL}{{\rm GL}}

\newcommand{\sa}{{\rm sa}}

\newcommand{\red}{{\rm r}}

\newcommand{\Ad}{{\rm Ad}}
\newcommand{\rb}{{\rm b}}
\newcommand{\distnu}{{\rm dist_{nu}}}
\newcommand{\dist}{{\rm dist}}
\newcommand{\prox}{{\rm prox}}
\newcommand{\UCP}{{\rm UCP}}
\newcommand{\rH}{{\rm H}}

\newcommand{\pa}{\|}

\newtheorem{theorem}{Theorem}[section]

\newtheorem{lemma}[theorem]{Lemma}
\newtheorem{proposition}[theorem]{Proposition}

\theoremstyle{definition}

\newtheorem{example}[theorem]{Example}

\begin{document}

\title{Metric aspects of noncommutative homogeneous spaces}
\dedicatory{Dedicated to Marc A. Rieffel \\
in honor of his seventieth birthday}
\author{Hanfeng Li}
\thanks{Partially supported by NSF Grant DMS-0701414.}
\address{Department of Mathematics \\
SUNY at Buffalo \\
Buffalo, NY 14260-2900, U.S.A.} \email{hfli@math.buffalo.edu}
\date{May 21, 2009}


\begin{abstract}
For a closed cocompact subgroup $\Gamma$ of a locally compact group $G$,
given a compact abelian subgroup $K$ of $G$ and a homomorphism $\rho:\hat{K}\rightarrow G$
satisfying certain conditions,
Landstad and Raeburn constructed equivariant
noncommutative deformations $C^*(\hat{G}/\Gamma, \rho)$
of the homogeneous space $G/\Gamma$, generalizing Rieffel's construction
of quantum Heisenberg manifolds. We show that when $G$ is a Lie group and $G/\Gamma$ is connected,
given any norm on the Lie algebra of $G$, the seminorm on $C^*(\hat{G}/\Gamma, \rho)$
induced by the derivation map of the canonical $G$-action defines a compact quantum metric.
Furthermore, it is shown that  this compact quantum metric space depends on $\rho$ continuously,
with respect to quantum Gromov-Hausdorff distances.
\end{abstract}

\maketitle


\section{Introduction}
\label{intro:sec}

In recent years, the quantum Heisenberg manifolds have received
quite some attention. These interesting $C^*$-algebras were constructed by Rieffel \cite{Rieffel89}
as deformation quantizations of the Heisenberg manifolds, and carry natural
actions of the Heisenberg group. The classification of these $C^*$-algebras
up to isomorphism (in most cases) and Morita equivalence (in all cases) has been achieved
by Abadie and her collaborators \cite{Abadie95, Abadie00, Abadie05, AE97}.
These $C^*$-algebras also appear in the work of Connes and Dubois-Violette on noncommutative $3$-spheres \cite{CD03, CD08}.

Aiming partly at giving a mathematical foundation for various approximations in the string theory,
such as the fuzzy spheres, namely the matrix algebras $M_n(\Cb)$, converging to the $2$-sphere $S^2$,
Rieffel developed a theory of compact quantum metric spaces and quantum Gromov-Hausdorff distance
between them \cite{Rieffel00, Rieffel01, Rieffel03}. As the information of the metric
on a compact metric space $X$ is encoded in the Lipschitz seminorm on the algebra of continuous
functions on  $X$, a quantum metric on (the compact quantum space represented by) a unital $C^*$-algebra $A$
is a (possibly $+\infty$-valued) seminorm on $A$ satisfying suitable conditions (see Section~\ref{action:sec} below
for detail).

One important class of examples of compact quantum metric spaces comes from ergodic actions
of a compact group $G$ on a unital $C^*$-algebra $A$, which should be thought of
as the translation action of $G$ on a noncommutative homogeneous space of $G$.
Given any length function on $G$,
such an ergodic action induces a quantum metric on $A$ \cite{Rieffel98} (see \cite{Li08} for
a generalization to ergodic actions of co-amenable compact quantum groups).
This class of examples includes the (fuzzy) spheres above and the noncommutative tori.
When $G$ is a compact connected Lie group and the length function comes from
the geodesic distance associated to some bi-invariant Riemannian metric on $G$,
this seminorm can also be defined in terms of the derivation map on the space
of once differentiable elements of $A$ with respect to the $G$-action \cite[Proposition 8.6]{Rieffel00}.
Explicitly, denote by $\sigma_X(b)$ the derivation of a once differentiable element $b$ of $A$ with respect to
an element $X$ of the Lie algebra $\fg$ of $G$ (see Section~\ref{der:sec} below for detail).
Then the seminorm $L(b)$ is defined as the norm of
the linear map $\fg\rightarrow A$ sending $X$ to $\sigma_X(b)$ when $b$ is once differentiable,
or $\infty$ otherwise.

It is natural to ask what conditions are needed to guarantee that $L$ defined above gives rise
to a quantum metric when $G$ is not compact. Rieffel raised the question about the quantum Heisenberg
manifolds in \cite{Rieffel03}. In \cite{Weaver} Weaver studied some sub-Riemannian metric on the quantum
Heisenberg manifolds, which does not quite fit into the above framework.
In \cite{Chakraborty} Chakraborty showed that certain seminorm associated to
some $\ell^1$-norm does define a quantum metric on the quantum Heisenberg manifolds. Since the $\ell^1$-norm
is bigger than the $C^*$-norm, this seminorm is bigger than the seminorm $L$ defined above. Thus
the result in \cite{Chakraborty} is weaker than what Riffel's question asks for.

Our first main result in this article is  an affirmative answer to Rieffel's question. In fact, we shall deal more generally
with Landstad and Raeburn's noncommutative homogeneous spaces.
In \cite{LR97} Landstad and Raeburn generalized Rieffel's construction to obtain equivariant deformations
of compact homogeneous spaces $G/\Gamma$, starting from a locally compact group $G$,
a closed cocompact subgroup $\Gamma$ of $G$, a compact abelian subgroup $K$ of $G$, and a homomorphism $\rho:\hat{K}\rightarrow G$
satisfying certain conditions. These $C^*$-algebras were denoted by $C^*_{\red}(\hat{G}/\Gamma, \rho)$
and were further studied in \cite{Landstad}. We shall see in Proposition~\ref{Fell:prop} below that
these algebras coincide with certain universal $C^*$-algebras, which we denote by $C^*(\hat{G}/\Gamma, \rho)$.
For our result to be valid for these algebras, we shall assume conditions (S1)-(S5) (see Sections~\ref{nh:sec}, \ref{der:sec}, and \ref{nondeformed:sec} below). Among these conditions, (S1)-(S3) are essentially the same but slightly weaker than the conditions
of Landstad and Raeburn. The conditions (S4) and (S5)  are just that $G$ is a Lie group and $G/\Gamma$ is connected.

\begin{theorem} \label{Lip:thm}
Let $G, \Gamma, K$ and $\rho$ satisfy the conditions (S1)-(S5).
Fix a norm on the Lie algebra $\fg$ of $G$.
Denote by $L_{\rho}$ the seminorm on $C^*(\hat{G}/\Gamma, \rho)$ defined above for
the canonical action $\alpha$ of $G$ on $C^*(\hat{G}/\Gamma, \rho)$.
Then $(C^*(\hat{G}/\Gamma, \rho), L_{\rho})$ is a $C^*$-algebraic compact
quantum metric space.
\end{theorem}

Since Rieffel introduced his quantum Gromov-Hausdorff distance in \cite{Rieffel00}, several variations
have appeared \cite{Kerr02, KL, Li06, Li08, Li07, Rieffel08, Wu}. Among these quantum distances, probably the most suitable one
in our current situation is the distance $\distnu$ discussed in \cite[Section 5]{KL}, which is the unital version of the quantum distance
introduced in \cite[Remark 5.5]{Li07}. As pointed out in \cite[Section 5]{KL}, this distance is no less
than the distances introduced in \cite{Kerr02, Rieffel00}. It is also no less than the distances in \cite{Rieffel08} (see Appendix below).
Our second main result says that the compact quantum metric spaces $ (C^*(\hat{G}/\Gamma, \rho), L_{\rho})$ depend on $\rho$ continuously.
Let us mention that among the conditions (S1)-(5), only the conditions (S1) and (S2) involve $\rho$.

\begin{theorem} \label{continuity:thm}
Fix $G$, $\Gamma$, and $K$ so that there exists $\rho$ satisfying the conditions
(S1)-(S5). Denote by $\Omega$ the set of all $\rho$ satisfying the conditions (S1)
and (S2), equipped with the weakest topology making the maps $\Omega\to G$
sending $\rho$ to $\rho(s)$ to be continuous for each $s\in \hat{K}$. Then
$\Omega$ is a locally compact metrizable space.
Fix a norm on the Lie algebra $\fg$ of $G$. Then for any $\rho'\in \Omega$,
$\distnu(C^*(\hat{G}/\Gamma, \rho), C^*(\hat{G}/\Gamma, \rho'))\to 0$
as $\rho\to \rho'$.
\end{theorem}

This paper is organized as follows. In Section~\ref{nh:sec} we recall
Landstad and Raeburn's construction of noncommutative homogeneous spaces,
and establish some general properties of these noncommutative spaces.
The relation between the derivations coming from two canonical group
actions on $C^*(\hat{G}/\Gamma, \rho)$ is established in Section~\ref{der:sec}.
In Section~\ref{nondeformed:sec} we show that in the nondeformed
case $L_{\rho}$ is essentially the Lipschitz seminorm corresponding to
some metric on $G/\Gamma$. A general result of establishing certain seminorm being a quantum metric
by the help of a compact group action is proved in Section~\ref{action:sec}.
Theorems~\ref{Lip:thm} and \ref{continuity:thm} are proved
in Sections~\ref{Lip:sec} and \ref{GH:sec} respectively. In an appendix
we compare the distance $\distnu$ and the proximity Rieffel introduced
in \cite{Rieffel08}.

\noindent{\it Acknowledgements.} I am grateful to Wei Wu for comments.

\section{Noncommutative homogeneous spaces}
\label{nh:sec}

In this section we recall Landstad and Raeburn's construction of noncommutative deformations of homogeneous spaces,
discuss some examples, and establish some general properties of these noncommutative homogeneous spaces. These properties
are of independent interest themselves.

Let $G$ be a locally compact group. Throughout this paper,
we make the following
standard assumptions:
\begin{enumerate}
\item[(S1)] $K$ is a compact abelian subgroup of $G$, and $\rho: \hat{K}\rightarrow
G$ is a group homomorphism from its Pontryagin dual $\hat{K}$ into $G$
such that $\rho(\hat{K})$ commutes with $K$.

\item[(S2)] $\Gamma$ is a closed subgroup of $G$ commuting with $K$ and
satisfies
\begin{eqnarray*}
\Omega_{\gamma}(s)&:=&\gamma\rho(s)\gamma^{-1}\rho(-s) \mbox{ is in } K  \mbox{ for all } s\in \hat{K}, \gamma\in \Gamma \mbox{ and }\\
\left<\Omega_{\gamma}(s), t\right>&=&\left<\Omega_{\gamma}(t), s\right> \mbox{ for all } s, t\in \hat{K}, \gamma\in \Gamma,
\end{eqnarray*}
where $\left< \cdot, \cdot\right>$ denotes the canonical pairing between $K$ and $\hat{K}$.
\end{enumerate}

Denote by $C_{\rb}(G)$ the Banach algebra of bounded continuous $\Cb$-valued functions on $G$, equipped with
the pointwise multiplication and the supremum norm.
Endow $K$ with its normalized Haar measure. Consider the
action of $K$ on $C_{\rb}(G)$ induced by the
right multiplication of $K$ on $G$.
For $f\in C_{\rb}(G)$, let $f_s\in C_{\rb}(G)$ for $s\in \hat{K}$ be the partial
Fourier transform defined by $f_s(x):=\int_K\overline{<k, s>}f(xk)\,dk$ for
$x\in G$ (this is denoted by $\hat{f}(x, s)$ in (1.3) of \cite{LR97}).
Note that although the action of $K$ on $C_{\rb}(G)$ may not be strongly continuous, we do have
$f_s\in C_{\rb}(G)$.
Then
\begin{eqnarray*}
C_{\rb, 1}(G):=\{f\in C_{\rb}(G)| \quad \pa f\pa_{\infty, 1}:=\sum_{s\in \hat{K}}\pa f_s\pa<\infty\}
\end{eqnarray*}
is a Banach $*$-algebra \cite[Proposition 5.2]{LR95}
with norm $\pa\cdot \pa_{\infty, 1}$ and operations
\begin{eqnarray} \label{b, 1:eq}
f*g(x)&=&\sum_{s, t}f_s(x\rho(t))g_t(x\rho(-s)), \\
\label{b, 1*:eq}
f^*(x)&=&\overline{f(x)}.
\end{eqnarray}
Fix a left invariant Haar measure on  $G$.
For each $s\in \hat{K}$ denote by $P_s$ the projection on $L^2(G)$ corresponding
to the restriction of the left regular representation $L|_K$ of $K$ in $L^2(G)$, i.e.,
\begin{eqnarray*}
P_s=\int_K\overline{\left<k, s\right>}L_k\, dk,
\end{eqnarray*}
where $L_y\xi(x)=\xi(y^{-1}x)$ for $\xi\in L^2(G)$, $x, y\in G$.
Then $C_{\rb,1}(G)$ has a faithful $*$-representation
$V$ on $L^2(G)$ \cite[Proposition 1.3]{LR97} given by
\begin{eqnarray} \label{mu:eq}
V(f)=\sum_{s, t}P_tL_{\rho(s)}M(f)L_{\rho(-t)}P_s,
\end{eqnarray}
where $M$ is the representation of $C_{\rb}(G)$ on $L^2(G)$ given
by $M(f)\xi(x)=f(x^{-1})\xi(x)$.
Denote by $C_0(G/\Gamma)$ the $C^*$-algebra of continuous $\Cb$-valued functions
on $G/\Gamma$ vanishing at $\infty$, and think of it as a $C^*$-subalgebra of $C_{\rb}(G)$
via the quotient map $G\rightarrow G/\Gamma$.
The space $C_{0, 1}(G/\Gamma, \rho):=C_0(G/\Gamma)\cap C_{\rb,1}(G, \rho)$
is a closed $*$-subalgebra of $C_{\rb,1}(G, \rho)$,
and the noncommutative homogeneous space
$C^*_{\red}(\hat{G}/\Gamma, \rho)$ of Landstad and Raeburn is defined as
the closure of $V(C_{0, 1}(G/\Gamma, \rho))$ \cite[Theorem 4.3]{LR97}.

Clearly the left translations $\alpha_y$ defined by $\alpha_y(f)(x)=f(y^{-1}x)$ for
$y\in G$
extend to isometric $*$-automorphisms of
$C_{0, 1}(G/\Gamma, \rho)$. They also extend to $*$-automorphisms
of $C^*_{\red}(\hat{G}/\Gamma, \rho)$ \cite[Theorem 4.3]{LR97}. We shall see later that this action of $G$ on $C^*_{\red}(\hat{G}/\Gamma, \rho)$
is strongly continuous.

Before discussing properties of these noncommutative homogeneous spaces, let us look at some examples.

\begin{example} \label{Quantum Heisenberg manifold}
Let $H_1$ be the $3$-dimensional Heisenberg group
consisting of matrices of the form

\begin{eqnarray*}
\begin{pmatrix} 1 & y & z  \\ 0 &  1 & x \\ 0 & 0 & 1
\end{pmatrix}
\end{eqnarray*}
as a subgroup of $\GL(3, \Rb)$. Denote by  $Z$ the subgroup consisting of
elements with
$x=y=0$ and $z\in \Zb$.
Then we can write the elements of $G:=H_1/Z$ as $(x, y, e^{2\pi i z})$
for $x, y, z\in \Rb$.
Fix a positive integer $c$.
Take
\begin{eqnarray*}
\Gamma=\{(x, y, e^{2\pi iz})\in G| x, y, cz\in \Zb\},\quad  K=\{(0, 0, e^{2\pi i z})\in G| z\in \Rb\}.
\end{eqnarray*}
Take $\mu, \nu \in \Rb$ and define $\rho:\Zb=\hat{K}\rightarrow G$ by
\begin{eqnarray*}
\rho(s)=(s\mu, s\nu, e^{\pi i s^2\mu\cdot\nu}).
\end{eqnarray*}
The $C^*$-algebra $C^*_{\red}(\hat{G}/\Gamma, \rho)$ is isomorphic to
Rieffel's quantum Heisenberg manifold $D_1$ in \cite[Theorem 5.5]{Rieffel89} (see \cite[page 493]{LR97}).
\end{example}

\begin{example}(cf. \cite[Example 4.17]{LR97}) \label{Higher Diemnsional Quantum Heisenberg manifold}
Let $H_n$ be the $2n+1$-dimensional Heisenberg group
consisting of matrices of the form

\begin{eqnarray*}
\begin{pmatrix} 1 & y_1 & y_2 &\cdots &y_n & z  \\ 0 &  1 & 0&\cdots & 0 &x_1 \\ 0 & 0 & 1& \cdots  & 0 & x_2\\
\cdots & \cdots & \cdots & \cdots & \cdots & \cdots
\\0 &0& 0&\cdots & 0& 1
\end{pmatrix}
\end{eqnarray*}
as a subgroup of $\GL(n+2, \Rb)$. Denote by  $Z$ the subgroup consisting of
elements with
$x_1= \cdots= x_n=y_1=\cdots=y_n=0$ and $z\in \Zb$.
Then we can write the elements of $G:=H_n/Z$ as $(x, y, e^{2\pi i z})$
for $x, y\in \Rb^n$ and $z\in \Rb$.
Fix positive integers $b_1, \dots, b_n, d_1, \dots, d_n$ and $c$
such that $b_jd_j|c$ for all $j$. Set $b=(b_1, \dots, b_n)$ and $d=(d_1, \dots, d_n)\in \Zb^n$.
Take
\begin{eqnarray*}
\Gamma=\{(x, y, e^{2\pi iz})\in G| b\cdot x, d\cdot y, cz\in \Zb\},\quad  K=\{(0, 0, e^{2\pi i z})\in G| z\in \Rb\}.
\end{eqnarray*}
Take $\mu, \nu \in \Rb^n$ and define $\rho:\Zb=\hat{K}\rightarrow G$ by
\begin{eqnarray*}
\rho(s)=(s\mu, s\nu, e^{\pi i s^2\mu\cdot\nu}).
\end{eqnarray*}
The $C^*$-algebra $C^*_{\red}(\hat{G}/\Gamma, \rho)$ is a higher-dimensional generalization
of Example~\ref{Quantum Heisenberg manifold}.
\end{example}

\begin{example}
Let $n\ge 3$.
Let $W$ be the subgroup of $\GL(n,\Zb)$
consisting of upper triangular matrices $(a_{j,l})$
with diagonal entries all being $1$.
Denote by  $Z$ the subgroup consisting of
matrices whose entries are all $0$ except diagonal ones being $1$ and
$a_{1,n}$ being an integer.
Then we can write the elements of $G:=W/Z$ as $(a_{j, l})$  with $a_{1, n}\in \Tb$.
Fix a positive integer $c$.
Take
\begin{eqnarray*}
\Gamma&=&\{(a_{j,l})\in G| a^c_{1, n}=1 \mbox{ and } a_{j, l}\in \Zb \mbox{ if } (j, l)\neq (1, n)\},\\
  K&=&\{(a_{j,l})\in G| a_{j,l}=0\quad  \mbox{ if }\quad  j<l \quad \mbox{ and } \quad (j,l)\neq(1,n)\}.
\end{eqnarray*}
Take $\mu, \nu \in \Rb$ and define $\rho:\Zb=\hat{K}\rightarrow G$ by
\begin{eqnarray*}
(\rho(s))_{j,l}=\left\{ \begin{array}{ll}
                   s\mu, & \mbox{ if }  (j,l)=(2,n), \\
                   s\nu, & \mbox{ if } (j, l)=(1, n-1), \\
                   e^{\pi i s^2\mu\cdot\nu}, &  \mbox{ if } (j, l)=(1, n), \\
                   0, & \mbox{ for other } j<l.
\end{array}
\right.
\end{eqnarray*}
For $n=3$ we get the quantum Heisenberg manifold in Example~\ref{Quantum Heisenberg manifold} again.
\end{example}

In the rest of this section  we establish some properties of $C^*_{\red}(\hat{G}/\Gamma, \rho)$.
Denote by $C^*(\hat{G}/\Gamma, \rho)$ the enveloping $C^*$-algebra of
the Banach $*$-algebra $C_{0, 1}(G/\Gamma, \rho)$ \cite[page 42]{Takesaki}.
By the universality of $C^*(\hat{G}/\Gamma, \rho)$ there is a canonical surjective
$*$-homomorphism $C^*(\hat{G}/\Gamma, \rho)\rightarrow
C^*_{\red}(\hat{G}/\Gamma, \rho)$ such that
the diagram
\begin{eqnarray*}
\xymatrix{
C_{0, 1}(G/\Gamma, \rho)  \ar[rd] \ar[r]
& C^*(\hat{G}/\Gamma, \rho) \ar[d]\\
& C^*_{\red}(\hat{G}/\Gamma, \rho)
}
\end{eqnarray*}
commutes.

Clearly
the right translations $\beta_k(f)(x)=f(xk)$ for $k\in K$
extend to isometric $*$-automorphisms of
$C_{0, 1}(G/\Gamma, \rho)$. Recall the action $\alpha$ of $G$ on $C_{0, 1}(G/\Gamma, \rho)$ defined before Example~\ref{Quantum Heisenberg manifold}.
Then $\alpha$ and $\beta$ induce actions of $G$ and $K$ on $C^*(\hat{G}/\Gamma, \rho)$ respectively, which we still denote
by $\alpha$ and $\beta$ respectively.
For each $s\in \hat{K}$, set
\begin{eqnarray} \label{Bs:eqn}
B_s:=\{f\in C_0(G/\Gamma)| f=f_s\}.
\end{eqnarray}

\begin{lemma} \label{cont:lemma}
The actions $\alpha$ and $\beta$ of $G$ and $K$ on
$C_{0, 1}(G/\Gamma, \rho)$ ($C^*(\hat{G}/\Gamma, \rho)$ resp.)
commute with each other and   are strongly continuous. The spectral spaces
$\{f\in C_{0, 1}(G/\Gamma, \rho)| \beta_k(f)=\left<k, s\right>f \mbox{ for all } k\in K\}$
and $\{a\in C^*(\hat{G}/\Gamma, \rho)| \beta_k(a)=\left<k, s\right>a \mbox{ for all } k\in K\}$
of $\beta$ corresponding to $s\in \hat{K}$ are exactly $B_s$, and the norm of $B_s$ in $C_{0, 1}(G/\Gamma, \rho)$
and $C^*(\hat{G}/\Gamma, \rho)$ is exactly the supremum norm.
\end{lemma}
\begin{proof}
Clearly $\alpha$ and $\beta$ commute with each other.
It is also clear that $B_s=\{f\in C_{0, 1}(G/\Gamma, \rho)| \beta_k(f)=\left<k, s\right>f \mbox{ for all } k\in K\}$
and that the norm of $B_s$ in $C_{0, 1}(G/\Gamma, \rho)$ is exactly the supremum norm.
It follows that the restrictions of the actions $\alpha$ and $\beta$ on
$B_s\subseteq C_{0, 1}(G/\Gamma, \rho)$ are strongly continuous for each $s\in \hat{K}$.
For any $f\in C_{0, 1}(G/\Gamma, \rho)$, one has $f_s\in B_s$ for each $s\in \hat{K}$.
For any $\varepsilon>0$ take a finite subset $F\subseteq \hat{K}$ such that $\sum_{s\in \hat{K}\setminus F}\pa f_s\pa<\varepsilon$.
Then $\pa f-\sum_{s\in F}f_s\pa_{\infty, 1}=\sum_{s\in \hat{K}\setminus F}\pa f_s\pa<\varepsilon$. Therefore $\oplus_{s\in \hat{K}}B_s$
is dense in $C_{0, 1}(G/\Gamma, \rho)$.  It follows that the actions $\alpha$ and $\beta$
are strongly continuous on $C_{0, 1}(G/\Gamma, \rho)$. Note that the canonical homomorphism
$C_{0, 1}(G/\Gamma, \rho)\rightarrow C^*(\hat{G}/\Gamma, \rho)$ is contractive \cite[Proposition 5.2]{Takesaki}. Consequently,
the induced actions of $\alpha$ and $\beta$ on $C^*(\hat{G}/\Gamma, \rho)$
are also strongly continuous.

Note that the subalgebra $B_0$ of $C_{0, 1}(G/\Gamma, \rho)$ is a $C^*$-algebra, which can be identified with
$C_0(G/K\Gamma)$. Since the natural homomorphism $C_{0, 1}(G/\Gamma, \rho)\rightarrow C^*_{\red}(\hat{G}/\Gamma, \rho)$ is
injective, so is the canonical homomorphism $C_{0, 1}(G/\Gamma, \rho)\rightarrow C^*(\hat{G}/\Gamma, \rho)$. As injective $*$-homomorphisms between
$C^*$-algebras are isometric, we conclude that the homomorphism of $B_0$ into $C^*(\hat{G}/\Gamma, \rho)$ is isometric.
For any $f\in B_s$ one has $f^**f\in B_0$ and the supremum norm of $f^**f$ is equal to the square of the supremum norm
of $f$. It follows that the homomorphism $C_{0, 1}(G/\Gamma, \rho)\rightarrow C^*(\hat{G}/\Gamma, \rho)$
is isometric on $B_s$. In particular, the image of $B_s$ in $C^*(\hat{G}/\Gamma, \rho)$ is closed.

Since the action $\beta$ of $K$ on $C^*(\hat{G}/\Gamma, \rho)$ is strongly continuous, the spectral space
$\{a\in C^*(\hat{G}/\Gamma, \rho)| \beta_k(a)=\left<k, s\right>a \mbox{ for all } k\in K\}$ is the image of the
continuous linear operator $C^*(\hat{G}/\Gamma, \rho)\rightarrow C^*(\hat{G}/\Gamma, \rho)$ sending
$a$ to $\int_K\overline{\left<k, s\right>}\beta_k(a)\,dk$. It follows that the image of $B_s=\{f\in C_{0, 1}(G/\Gamma, \rho)| \beta_k(f)=\left<k, s\right>f \mbox{ for all } k\in K\}$ in $C^*(\hat{G}/\Gamma, \rho)$
is dense in $\{a\in C^*(\hat{G}/\Gamma, \rho)| \beta_k(a)=\left<k, s\right>a \mbox{ for all } k\in K\}$.
Therefore the image of $B_s$ in $C^*(\hat{G}/\Gamma, \rho)$ is exactly $\{a\in C^*(\hat{G}/\Gamma, \rho)| \beta_k(a)=\left<k, s\right>a \mbox{ for all } k\in K\}$.
\end{proof}

We refer the reader to \cite[Chapter 2]{BO} for the basics of nuclear $C^*$-algebras.

\begin{proposition} \label{nuclear:prop}
The $C^*$-algebra $C^*(\hat{G}/\Gamma, \rho)$ is nuclear.
\end{proposition}
\begin{proof}
By Lemma~\ref{cont:lemma} the action $\beta$ of $K$ on $C^*(\hat{G}/\Gamma, \rho)$
is strongly continuous, and its fixed-point subalgebra
is $B_0$, a commutative $C^*$-algebra, and hence is nuclear \cite[Proposition 2.4.2]{BO}.
For any $C^*$-algebra carrying a strongly continuous action of a compact group, the algebra is nuclear if and only if
the fixed-point subalgebra is nuclear \cite[Proposition 3.1]{DLRZ}. Consequently, $C^*(\hat{G}/\Gamma, \rho)$ is nuclear.
\end{proof}

We shall need the following well-known fact a few times (see for example \cite[Proposition 4.5.1]{BO}).

\begin{lemma} \label{inj:lemma}
Let $H$ be a compact group, and let $\sigma_j$ be
a strongly continuous action of $H$ on a $C^*$-algebra $A_j$
for $j=1, 2$. Let $\varphi:A_1\rightarrow A_2$ be
an $H$-equivariant $*$-homomorphism. Then
$\varphi$ is injective if and only if the restriction
of $\varphi$ on the fixed-point subalgebra $A^H_1$ is injective.
In particular, if $\varphi$ is surjective and
$\varphi|_{A^H_1}$ is injective, then $\varphi$ is
an isomorphism.
\end{lemma}

\begin{proposition} \label{Fell:prop}
The canonical $*$-homomorphism $C^*(\hat{G}/\Gamma, \rho)\rightarrow
C^*_{\red}(\hat{G}/\Gamma, \rho)$ is an isomorphism.
\end{proposition}
\begin{proof}
We shall apply Lemma~\ref{inj:lemma} to show that the canonical $*$-homomorphism $\varphi: C^*(\hat{G}/\Gamma, \rho)\rightarrow
C^*_{\red}(\hat{G}/\Gamma, \rho)$ is
an isomorphism.
By \cite[Lemma 4.4]{LR97}
the action $\beta$ on $C_{0, 1}(G/\Gamma, \rho)$ extends
to an action of $K$ on $C^*_{\red}(\hat{G}/\Gamma, \rho)$, which we denote by $\beta'$.
Clearly
$\varphi$ is $K$-equivariant.
By Lemma~\ref{cont:lemma} $\beta$ is strongly continuous
on $C^*(\hat{G}/\Gamma, \rho)$. Since $\varphi$ is contractive, it follows that  $\beta'$ is strongly continuous
on $C^*_{\red}(\hat{G}/\Gamma, \rho)$. By Lemma~\ref{cont:lemma} the fixed-point
subalgebra $(C^*(\hat{G}/\Gamma, \rho))^K$
is $B_0$.
Since
the homomorphism $C_{0, 1}(G/\Gamma, \rho)\rightarrow
C^*_{\red}(\hat{G}/\Gamma, \rho)$ is injective, we see that
the restriction of $\varphi$ on $(C^*(\hat{G}/\Gamma, \rho))^K$  is injective. Therefore
the conditions of Lemma~\ref{inj:lemma} are satisfied and we conclude that $\varphi$ is
an isomorphism.
\end{proof}


We refer the reader to \cite{DF} for a comprehensive treatment
of {\it $C^*$-algebraic bundles}, which are usually called {\it Fell bundles} now.
Notice that for $f_s\in B_s$ and $g_t\in B_t$ the
product $f_s*g_t$ is in $B_{s+t}$ and $f^*_s$ is in
$B_{-s}$. Also $||f^*_s*f_s||=||f_s||^2$. Therefore
we have a Fell bundle $\mathcal{B}^{\rho}=\{B_s\}_{s\in \hat{K}}$
over $\hat{K}$ with operations given by
(\ref{b, 1:eq}) and (\ref{b, 1*:eq}).
It is easy to see that $C_{0, 1}(G/\Gamma, \rho)$ is
exactly the $L^1$-algebra of $\mathcal{B}^{\rho}$ (cf. the proof
of \cite[Proposition 5.2]{LR95}). Thus the $C^*$-algebra $C^*(\hat{G}/\Gamma, \rho)$
is also the enveloping
$C^*$-algebra $C^*(\mathcal{B}^{\rho})$ of the Fell
bundle $\mathcal{B}^{\rho}$.

Next we discuss what happens if we let $\rho$ vary continuously. We refer the reader to \cite[Chapter 10]{Dixmier} for
the basics of continuous fields of Banach spaces and $C^*$-algebras.
On page 505 of \cite{LR97} Landstad and Raeburn pointed out that it seems reasonable that
we shall
get a continuous field of $C^*$-algebras, but no proof was given there.
This is indeed true, and we give a proof here. To be precise, fix $G$, $\Gamma$ and $K$,
let $W$ be a locally compact Hausdorff space and for each $w\in W$
we assign a $\rho_w$ satisfying (S1) and (S2) such that the map
$w\mapsto \rho_w(s)$ is continuous for each $s\in \hat{K}$.  Notice
that $\mathcal{B}^{\rho}$ as a Banach space bundle over $\hat{K}$ do not
depend on $\rho$. For clarity we denote the product
and $*$-operation in (\ref{b, 1:eq}) and (\ref{b, 1*:eq})
by $f_s*_wg_t$ and $f^{*_w}_s$. For any $f_s\in B_s$ and $g_t\in B_t$, clearly
the maps $w\mapsto f_s*_wg_t$ and $w\mapsto f^{*_w}_s$ are both continuous.
This leads to the next lemma, which is a slight generalization
of \cite[Proposition 3.3, Theorem 3.5]{AE01}.
The proof of \cite[Proposition 3.3, Theorem 3.5]{AE01}, which in turn follows the lines of
\cite{Rieffel93},
 is easily
seen to hold also in our case.

\begin{lemma} \label{cont of Fell:lemma}
Let $H$ be a discrete group and $A_h$ be a vector space
for each $h\in H$.
Let $W$ be a locally compact Hausdorff space
and for each $w\in W$ assign norms and algebra operations making $\mathcal{A}^w=\{A_h\}_{h\in H}$
into
a Fell bundle in such a way that
for any $f_s\in A_s$ and $g_t\in A_t$ the map $w\mapsto ||f_s||_w\in \Rb$ is continuous
(then we have a continuous field of Banach spaces $(A_s, ||\cdot ||_w)_{w\in W}$
over $W$ for each $s\in H$)
and the sections
$w\mapsto f_s*_wg_t\in B_{st}$ and
$w\mapsto f^{*_w}_s\in B_{s^{-1}}$ are continuous in the above continuous fields
of Banach spaces $(B_{st}, ||\cdot ||_w)_{w\in W}$ and $(B_{s^{-1}}, ||\cdot ||_w)_{w\in W}$ respectively .
Then the map $w\mapsto ||f||_w$ is upper semi-continuous for each $f\in \oplus_{s\in H}A_s$,
where $||\cdot ||_w$ is the norm on the enveloping $C^*$-algebra $C^*(\mathcal{A}^w)$ and
extends the norm of $A_s$ as part of $\mathcal{A}^w$ for each $s\in H$.
Moreover, if $H$ is amenable, then $\{C^*(\mathcal{A}^w)\}_{w\in W}$
is a continuous field of $C^*$-algebras with the field structure determined
by the continuous sections $w\mapsto f$ for all $f\in \oplus_{s\in H}A_s$.
\end{lemma}

Since every discrete abelian group is amenable \cite[page 14]{Paterson}, from Proposition~\ref{Fell:prop}
we get

\begin{proposition} \label{cont field:prop}
Fix $G, \Gamma$ and $K$. Let $W$ by a locally compact Hausdorff space and for each $w\in W$
let $\rho_w$ satisfy (S1) and (S2) such that the map
$w\mapsto \rho_w(s)$ is continuous for each $s\in \hat{K}$.
Then $\{C^*(\hat{G}/\Gamma, \rho_w)\}_{w\in W}$
is a continuous field of $C^*$-algebras with the field structure determined
by the continuous sections $w\mapsto f$ for all $f\in \oplus_{s\in \hat{K}}B_s$.
\end{proposition}

\section{Derivations}  \label{der:sec}

In this section we prove Proposition~\ref{der:prop}, to establish the relation between derivations
coming from $\alpha$ and $\beta$.

Throughout the rest of this paper, we assume:
\begin{enumerate}
\item[(S3)] $G/\Gamma$ is compact.
\item[(S4)] $G$ is a Lie group.
\end{enumerate}
The examples in Section~\ref{nh:sec} all satisfy these conditions.

We refer the reader to \cite[Section 1.3]{Hamilton} for the discussion about differentiable maps into
Fr\'echet spaces. We just recall that a continuous map $\psi$ from a smooth manifold $M$ into a Fr\'echet space
$A$  is {\it continuously differentiable}
if for any chart $(U, \phi)$ of $G$, where $U$ is an open subset of some Euclidean space $\Rb^n$
and
$\phi$ is a diffeomorphism from $U$ onto an open set of $M$, the derivative
$$ D(\psi\circ \phi)(x, h)=\lim_{\Rb \ni\nu\to 0}\frac{\psi\circ \phi(x+\nu h)-\psi\circ \phi(x)}{\nu}$$
exists for all $(x, h)\in (U, \Rb^n)$ and is a jointly continuous map from $(U, \Rb^n)$ into $A$.
In such case, $D(\psi\circ \phi)(x, h)$ is linear on $h$, and depends only on $\psi$ and the tangent vector $u:=\phi_*(v_{x, h})$
of $M$ at $\phi(x)$, where $v_{x, h}$ denotes the tangent vector $h$ at $x$. Thus we may denote $D(\psi\circ \phi)(x, h)$ by
$\partial_u\psi$.  Then $\partial_u\psi$ is linear on $u$.

Denote by $\fg$ and $\fk$ the Lie algebras
of $G$ and $K$ respectively.
For a strongly continuous action $\sigma$ of $G$ on a Banach space $A$ as isometric automorphisms,
we say that an element $a\in A$ is {\it once differentiable} with respect to $\sigma$ if
the orbit map $\psi_a$ from $G$ into $A$ sending $x$ to $\sigma_x(a)$ is continuously differentiable.
Then the set $A_1$ of once differentiable elements is a linear subspace of $A$.
For any $a\in A$ and any compactly supported smooth $\Cb$-valued function $\varphi$ on $G$, it is easily checked
that $\int_G\varphi(x)\sigma_x(a)\, dx$ is in $A_1$. As $a$ can be approximated by such elements, we see
that $A_1$ is dense in $A$.
Thinking of $\fg$ as the tangent space of $G$ at the identity element, for each $X\in \fg$ we have the linear map
$\sigma_X: A_1\rightarrow A$ sending $a$ to $\partial_X\psi_a$.  Fix a norm on $\fg$. We define a seminorm
$L$ on $A_1$ by setting $L(a)$ to be the norm of the linear map $\fg\rightarrow A$ sending $X$ to $\sigma_Xa$.

\begin{lemma} \label{derivation via length:lemma}
Let $\sigma$ be a strongly continuous action of $G$  on a Banach space $A$ as isometric automorphisms.
For any $a\in A_1$, one has
$$ L(a)=\sup_{0\neq X\in \fg} \frac{\pa \sigma_{e^{X}}(a)-a\pa}{\pa X\pa}.$$
\end{lemma}
\begin{proof}The proof is similar to that of \cite[Proposition 8.6]{Rieffel00}.
Let $X\in \fg$ with $\pa X\pa =1$. One has
\begin{eqnarray*}
\sup_{\nu >0} \frac{\pa \sigma_{e^{\nu X}}(a)-a\pa}{\nu }\ge \lim_{\nu\to 0^+} \frac{\pa \sigma_{e^{\nu X}}(a)-a\pa}{\nu }=\pa \sigma_X(a)\pa.
\end{eqnarray*}
For any $\nu>0$, one also has
\begin{eqnarray*}
\pa \sigma_{e^{\nu X}}(a)-a\pa
&=&\pa \int^{\nu}_0\sigma_{e^{z X}}(\sigma_X(a))\, dz\pa
\le  \int^{\nu}_0\pa \sigma_{e^{z X}}(\sigma_X(a))\pa \, dz \\
&=& \int^{\nu}_0\pa \sigma_X(a)\pa \, dz
=\nu \pa \sigma_X(a)\pa.
\end{eqnarray*}
Therefore
$$ \sup_{\nu >0} \frac{\pa \sigma_{e^{\nu X}}(a)-a\pa}{\nu }=\pa \sigma_X(a)\pa.$$
Thus
\begin{eqnarray*}
\sup_{0\neq X\in \fg} \frac{\pa \sigma_{e^{X}}(a)-a\pa}{\pa X\pa}
&=&\sup_{X\in \fg, \pa X\pa=1} \sup_{\nu >0} \frac{\pa \sigma_{e^{\nu X}}(a)-a\pa}{\nu } \\
&=&\sup_{X\in \fg, \pa X\pa=1} \pa \sigma_X(a)\pa
=L(a).
\end{eqnarray*}
\end{proof}

\begin{lemma} \label{C1:lemma}
Let $\sigma$ be a strongly continuous action of $G$  on a Banach space $A$ as isometric automorphisms.
Then $A_1$ is a Banach space with the norm $\fp(a):=L(a)+\pa a\pa$. Suppose that $\sigma'$ is
a strongly continuous  isometric action of a topological group
$H$ on $A$, commuting with $\sigma$. Then $H$ preserves $A_1$, and the restriction
of $\sigma'$ on $A_1$ preserves the norm $\fp$ and is strongly continuous with respect to $\fp$.
\end{lemma}
\begin{proof} Let $\{a_n\}_{n\in \Nb}$ be a Cauchy sequence in $A_1$ under the norm $\fp$.
Then as $n$ goes to infinity, $a_n$ converges to some $a\in A$, and $\sigma_X(a_n)$ converge to some $b_X$ in $A$ uniformly
on $X$ in bounded subsets of $\fg$.
Let $\varrho: [0, 1]\rightarrow G$ be a continuously differentiable curve in $G$.
Then $\lim_{z\rightarrow 0}\frac{\sigma_{\varrho_{\nu +z}}(a_n)-\sigma_{\varrho_{\nu}}(a_n)}{z}=
\sigma_{\varrho_{\nu}}(\sigma_{\varrho'_{\nu}}(a_n))$ for all $\nu \in [0, 1]$.
Thus
$$\sigma_{\varrho_{\nu}}(a_n)-\sigma_{\varrho_{0}}(a_n)=\int^{\nu}_0\sigma_{\varrho_{z}}(\sigma_{\varrho'_{z}}(a_n))\, dz.$$
Letting $n\to \infty$ we get
$$ \sigma_{\varrho_{\nu}}(a)-\sigma_{\varrho_{0}}(a)=\int^{\nu}_0\sigma_{\varrho_{z}}(b_{\varrho'_{z}})\, dz.$$
Therefore
$\lim_{z\rightarrow 0}\frac{\sigma_{\varrho_{ z}}(a)-\sigma_{\varrho_{0}}(a)}{z}=\sigma_{\varrho_0}(b_{\varrho'_{0}})$.
It follows easily that $a\in A_1$ and $\sigma_X(a)=b_X$ for all $X\in \fg$. Consequently,
$a_n$ converges to $a$ in $A_1$ under the norm $\fg$, and hence $A_1$ is a Banach space under the norm
$\fg$.

Clearly $\sigma'$ preserves $A_1$ and the norm $\fp$. For any $a\in A_1$, the set of $\sigma_X(a)$ for $X$ in the unit ball
of $\fg$ is compact. Then for any $h\in H$ and $\varepsilon>0$, when $h'\in H$ is close enough to
$h$, one has $\pa \sigma'_h(a)-\sigma'_{h'}(a)\pa <\varepsilon$ and $\pa \sigma_X(\sigma'_h(a))-\sigma_X(\sigma'_{h'}(a))\pa
=\pa \sigma'_h(\sigma_X(a))-\sigma'_{h'}(\sigma_X(a)) \pa<\varepsilon$ for all $X$ in the unit ball of $\fg$.
Consequently,  $\fp(\sigma'_h(a)-\sigma'_{h'}(a))
=L(\sigma'_h(a)-\sigma'_{h'}(a))+\pa \sigma'_h(a)-\sigma'_{h'}(a)\pa< 2\varepsilon$. Therefore the restriction of
$\sigma'$ on $A_1$ is strongly continuous with respect to $\fp$.
\end{proof}

By Lemma~\ref{cont:lemma} the actions $\alpha$ and $\beta$ on $C^*(\hat{G}/\Gamma, \rho)$ commute with each other and
are strongly continuous.
Denote by $C^1(\hat{G}/\Gamma, \rho)$ the space of once differentiable elements
of $C^*(\hat{G}/\Gamma, \rho)$ with respect to the action $\alpha$.
Recall the $B_s$ defined in (\ref{Bs:eqn}).

\begin{proposition} \label{der:prop}
Let $X_1, \cdots{}, X_n$ be a basis of $\fg$. For
$Y\in \fk$ say
\begin{eqnarray*}
\Ad_x(Y)=\sum_jF_{j, Y}(x)X_j,
\end{eqnarray*}
where $\Ad$ denotes the adjoint action of $G$ on $\fg$.
Then $F_{j, Y}\in B_0$.
Any $f\in C^1(\hat{G}/\Gamma, \rho)$ is once differentiable with respect to
the action $\beta$ and
\begin{eqnarray} \label{der:eq}
\beta_Y(f)=-\sum_jF_{j, Y}*\alpha_{X_j}(f).
\end{eqnarray}
\end{proposition}
\begin{proof}
Clearly $F_{j, Y}$ is a smooth function on $G$.
Since the subgroups $\Gamma$, $K$ and $\rho(\hat{K})$ commute with $K$, if
$y$ is in any of these subgroups, then
$\Ad_y(Y)=Y$, and hence
$$ \sum_jF_{j, Y}(x)X_j=\Ad_x(Y)=\Ad_x(\Ad_y(Y))=\Ad_{xy}(Y)=\sum_jF_{j, Y}(xy)X_j,$$
which means that $F_{j, Y}$ is invariant under the right translation of $y$.
Thus $F_{j, Y}\in C(G/K\Gamma)=C_0(G/K\Gamma)=B_0$. For each $X\in \fg$ denote by
$X^{\#}$ ($X_{\#}$ resp.) the corresponding right (left resp.)
translation invariant vector field on $G$. Then $Y_{\#}=\sum_jF_{j, Y}X^{\#}_j$.

Let $f\in C^1(\hat{G}/\Gamma, \rho)\cap B_s$ for some $s\in \hat{K}$.
By Lemma~\ref{cont:lemma} the norm on $B_s\subseteq C^*(\hat{G}/\Gamma, \rho)$ is exactly the
supremum norm. Thus
$f$ belongs to the space $C^1(G)$ of continuously differentiable functions
on $G$.
For any continuous vector field $Z$ on $G$ denote by $\partial_Z$ the corresponding
derivation map $C^1(G)\rightarrow C(G)$. Then
\begin{eqnarray*}
\partial_{Y_{\#}}(f)=\sum_jF_{j, Y}\partial_{X^{\#}_j}(f)=-\sum_jF_{j, Y}\alpha_{X_j}(f).
\end{eqnarray*}
Since $F_{j, Y}$ is invariant under the right
translation of $\Gamma$ and $\rho(K)$, we have $F_{j, Y}(x)g_t(x)=F_{j, Y}*g_t(x)$ for any $g_t\in B_t$
and $x\in G$. By Lemma~\ref{cont:lemma} the actions $\alpha$ and $\beta$ on $C^*(\hat{G}/\Gamma, \rho)$
commute with each other. Thus $\alpha$ preserves $B_s$, and hence $\alpha_X(f)\in B_s$ for every $X\in \fg$.
Therefore $\partial_{Y_{\#}}(f)=-\sum_jF_{j, Y}*\alpha_{X_j}(f)$.

Let $\varrho: [0, 1]\rightarrow K$ be a continuously differentiable curve in $K$.
Then
$$ \lim_{z\to 0}\frac{f(x\varrho_{\nu+z})-f(x\varrho_{\nu})}{z}=(\partial_{(\varrho'_{\nu})_{\#}}(f))(x\varrho_{\nu})=(-\sum_jF_{j, \varrho'_{\nu}}*\alpha_{X_j}(f))(x\varrho_{\nu})$$
for all $\nu \in [0, 1]$ and $x\in G$, and hence we have the
integral form
\begin{eqnarray} \label{der2:eq}
f(x\varrho_{\nu})-f(x\varrho_0)=\int^{\nu}_0(-\sum_jF_{j, \varrho'_{z}}*\alpha_{X_j}(f))(x\varrho_z)\, dz
\end{eqnarray}
for all $\nu \in [0, 1]$ and $x\in G$.
The left hand side of (\ref{der2:eq}) is the value of
$\beta_{\varrho_{\nu}}(f)-\beta_{\varrho_0}(f)$  at $x$,
while the right hand side of (\ref{der2:eq}) is the value of
$\int^{\nu}_0\beta_{\varrho_z}(-\sum_jF_{j, \varrho'_z}*\alpha_{X_j}(f))\, dz$
at $x$, where the integral is
taken in $B_s\subseteq C^*(\hat{G}/\Gamma, \rho)$. Therefore
\begin{eqnarray} \label{der3:eq}
\beta_{\varrho_{\nu}}(f)-\beta_{\varrho_0}(f)=\int^{\nu}_0\beta_{\varrho_z}(-\sum_{j}F_{j, \varrho'_z}*\alpha_{X_j}(f))\, dz
\end{eqnarray}
for all $\nu \in [0, 1]$.

Clearly (\ref{der3:eq}) also holds for $f\in \oplus_{s\in \hat{K}}(C^1(\hat{G}/\Gamma, \rho)\cap B_s)$.
By Lemma~\ref{C1:lemma} $C^1(\hat{G}/\Gamma, \rho)$ is a Banach space
with norm $\fp(\cdot)=L(\cdot)+\pa \cdot \pa$,  $\beta$ preserves $C^1(\hat{G}/\Gamma, \rho)$ and $\fp$, and the restriction of $\beta$
on $C^1(\hat{G}/\Gamma, \rho)$ is strongly continuous
on $C^1(\hat{G}/\Gamma, \rho)$ with respect to $\fp$.
By Lemma~\ref{cont:lemma} the spectral subspace of $ C^*(\hat{G}/\Gamma, \rho)$ corresponding
to $s\in \hat{K}$ for the action $\beta$ is equal to $B_s$.
It follows  that the spectral subspace of $ C^1(\hat{G}/\Gamma, \rho)$ corresponding
to $s\in \hat{K}$ for the restriction of $\beta$ on $C^1(\hat{G}/\Gamma, \rho)$ is exactly $C^1(\hat{G}/\Gamma, \rho)\cap B_s$.
Then standard techniques tell us that $\oplus_{s\in \hat{K}}(C^1(\hat{G}/\Gamma, \rho)\cap B_s)$
is dense in $C^1(\hat{G}/\Gamma, \rho)$ with respect to $\fp$. Notice that both sides of (\ref{der3:eq})
define continuous maps from $C^1(\hat{G}/\Gamma, \rho)$ to $C^*(\hat{G}/\Gamma, \rho)$.
Therefore (\ref{der3:eq}) holds for all $f\in C^1(\hat{G}/\Gamma, \rho)$.
Consequently,
$$\lim_{z\to 0}\frac{\beta_{\varrho_z}(f)-\beta_{\varrho_0}(f)}{z}=\beta_{\varrho_0}(-\sum_jF_{j, \varrho'_0}*\alpha_{X_j}(f))$$
for all $f\in C^1(\hat{G}/\Gamma, \rho)$. It follows easily that
$f$ is once differentiable with respect to $\beta$ and $\beta_Y(f)=-\sum_jF_{j, Y}*\alpha_{X_j}(f)$ for
all $f\in C^1(\hat{G}/\Gamma, \rho)$ and $Y\in \fk$.
\end{proof}

We shall need the following lemma (compare \cite[Proposition 2.5]{Rieffel07}).

\begin{lemma} \label{approximate:lemma}
Let $\sigma$ be a strongly continuous action of $G$ on a Banach space $A$ as isometric automorphisms.
Let $a\in A$.
Then for any $\varepsilon>0$, there is some $b\in A$ such that $b$ is smooth with respect to $\sigma$,
$\pa b\pa \le \pa a\pa$, $\pa b-a\pa \le \varepsilon$, and
$\sup_{0\neq X\in \fg}\frac{\pa \sigma_X(b)\pa}{\pa X\pa}\le \sup_{0\neq X\in \fg}\frac{\pa \sigma_{e^X}(a)-a\pa}{\pa X\pa}$. If $A$ has an isometric involution being invariant under $\sigma$, then
when $a$ is self-adjoint, we can choose $b$ also to be self-adjoint.
\end{lemma}
\begin{proof} Endow $G$ with a left-invariant Haar measure.
Let $U$ be a small open neighborhood of the identity element in $G$ with compact closure, which
we shall determine later. Let $\varphi$ be a non-negative smooth function on $G$ with support contained in $U$
such that $\int_G \varphi(x)\, dx=1$. Set $b=\int_G \varphi(x)\sigma_x(a)\, dx$. Then $b$ is smooth with respect to
$\sigma$, and $\pa b\pa \le \pa a\pa$. When $U$ is small enough,
we have $\pa a-b\pa\le \varepsilon/2$. For any $X\in \fg$, setting $\psi(x)=\Ad_{x^{-1}}(X)$, we have
\begin{eqnarray*}
\pa \sigma_{e^X}(b)-b\pa
&=& \pa \int_G\varphi(x)(\sigma_{e^Xx}(a)-\sigma_x(a))\, dx\pa \\
&=&\pa \int_G \varphi(x)\sigma_x(\sigma_{e^{\psi(x)}}(a)-a)\, dx\pa \\
&\le & \int_G \varphi(x)\pa \sigma_x(\sigma_{e^{\psi(x)}}(a)-a)\pa \, dx\\
&\le & \sup_{x\in U} \pa \sigma_{e^{\psi(x)}}(a)-a\pa \\
&\le & \sup_{0\neq Y\in \fg}\frac{\pa \sigma_{e^Y}(a)-a\pa}{\pa Y\pa}\cdot \sup_{x\in U} \pa \psi(x)\pa.
\end{eqnarray*}
Set $\delta=\varepsilon/(2+2\pa a\pa)$.
When $U$ is small enough, we have $\pa \Ad_{x^{-1}}(X)\pa \le (1+\delta)\pa X\pa $ for all $X\in \fg$ and $x\in U$.
Then $\pa \sigma_{e^X}(b)-b\pa \le (1+\delta)\pa X\pa \sup_{0\neq Y\in \fg}\frac{\pa \sigma_{e^Y}(a)-a\pa}{\pa Y\pa}$
for all $X\in \fg$. By Lemma~\ref{derivation via length:lemma} we get
$$\sup_{0\neq X\in \fg}\frac{\pa \sigma_X(b)\pa}{\pa X\pa}=\sup_{0\neq X\in \fg}\frac{\pa \sigma_{e^X}(b)-b\pa}{\pa X\pa}\le (1+\delta)\sup_{0\neq X\in \fg}\frac{\pa \sigma_{e^X}(a)-a\pa}{\pa X\pa}.$$
Now it is clear that $b'=b/(1+\delta)$ satisfies the requirement.
Note that $b'$ is self-adjoint if $a$ is so.
\end{proof}

\section{Nondeformed case} \label{nondeformed:sec}

In this section we consider the nondeformed case, i.e.,
the case $\rho$ is the trivial homomorphism $\rho_0$ sending the whole $\hat{K}$ to the identity
element of $G$. In Proposition~\ref{Lipschitz:prop} we identify $L_{\rho_0}$ on $C^1(\hat{G}/\Gamma, \rho_0)$ with
the Lipschitz seminorm for certain metric on $G/\Gamma$.

Note that $C_{0, 1}(G/\Gamma, \rho_0)$ is sub-$*$-algebra of $C_0(G/\Gamma)=C(G/\Gamma)$.
By the universality of $C^*(\hat{G}/\Gamma, \rho_0)$ we have  a natural
$*$-homomorphism $\psi$ of $C^*(\hat{G}/\Gamma, \rho_0)$ into $C(G/\Gamma)$, extending the inclusion
$C_{0, 1}(G/\Gamma, \rho_0) \hookrightarrow C(G/\Gamma)$. The right translation of $K$ on $G$ induces
a strongly continuous action $\beta''$ of $K$ on $C(G/\Gamma)$, and clearly $\psi$ intertwines $\beta$ and $\beta''$.
An application of Lemmas~\ref{inj:lemma} and \ref{cont:lemma} tells us that $\psi$ is injective.
By definition $B_s$ is the spectral subspace of $C(G/\Gamma)$ corresponding to $s\in \hat{K}$. Thus
$\oplus_{s\in \hat{K}}B_s$ is dense in $C(G/\Gamma)$. As $\oplus_{s\in \hat{K}}B_s$ is in the image of $\psi$,
we see that $\psi$ is surjective and hence is an isomorphism. We shall identify $C^*(\hat{G}/\Gamma, \rho_0)$
and $C(G/\Gamma)$ via $\psi$.

The seminorm $L_{\rho_0}$ describes the size of derivatives of $f\in C^1(\hat{G}/\Gamma, \rho_0)$. If it corresponds to
some metric on $G/\Gamma$, this metric should be kind of geodesic distance. In order for the geodesic distance to be defined,
throughout the rest of this paper we assume:
\begin{enumerate}
\item[(S5)] $G/\Gamma$ is connected.
\end{enumerate}
The examples in Section~\ref{nh:sec} all satisfy this condition.

Fix an inner product on $\fg$. Then we obtain a right translation invariant Riemannian metric
on $G$ in the usual way. Denote by $d_G$ the geodesic distance on connected components
of $G$. We extend $d_G$ to a semi-distance on $G$ via setting $d_G(x, y)=\infty$ if
$x$ and $y$ lie in different connected components of $G$.

\begin{lemma} \label{quotient metric:lemma}
The function $d$ on $G/\Gamma\times G/\Gamma$ defined by $ d(x\Gamma, y\Gamma):=\inf_{x'\in x\Gamma, y'\in y\Gamma} d_G(x', y')$
is equal to $\inf_{y'\in y\Gamma} d_G(x, y')$. It is a metric on $G/\Gamma$ and induces the quotient topology on $G/\Gamma$.
\end{lemma}
\begin{proof}
Let $V$ be a connected component of $G$. Then $V\Gamma$ is clopen in $G$, and hence $V\Gamma/\Gamma$ is clopen in $G/\Gamma$ for the
quotient topology. As $G/\Gamma$ is connected, we conclude that $V\Gamma/\Gamma=G/\Gamma$. Therefore $d$ is finite valued.

Since $d_G$ is right translation invariant, we have $\inf_{x'\in x\Gamma, y'\in y\Gamma} d_G(x', y')=\inf_{y'\in y\Gamma} d_G(x, y')$.
It follows easily that $d$ is a metric on $G/\Gamma$.

Let $x\in G$. Let $W$ be a neighborhood of $x\Gamma$ in $G/\Gamma$ for the quotient topology.
Then there exists $\varepsilon>0$ such that if $d_G(x, y)<\varepsilon$, then $y\Gamma \in W$.
It follows that if $d(x\Gamma, y\Gamma)<\varepsilon$, then $y\Gamma \in W$.
Therefore the topology induced by $d$ on $G/\Gamma$ is finer than the quotient topology.
For any $\varepsilon'>0$, set $U=\{y\in G| d_G(x, y)<\varepsilon'\}$. Then $U$ is an open neighborhood
of $x$. Thus $U\Gamma/\Gamma$ is an open neighborhood of $x\Gamma$ for the quotient topology.
For any $z\Gamma\in U\Gamma/\Gamma$, we can find $z'\in z\Gamma\cap U$ and hence $d(x\Gamma, z\Gamma)\le
d_G(x, z')<\varepsilon'$. Therefore the quotient topology on $G/\Gamma$ is finer than
the topology induced by $d$. We conclude that $d$ induces the quotient topology.
\end{proof}

\begin{proposition} \label{Lipschitz:prop}
For any $f\in C^1(\hat{G}/\Gamma, \rho_0)\subseteq C^*(\hat{G}/\Gamma, \rho_0)=C(G/\Gamma)$, we have
$$ L_{\rho_0}(f)=\sup_{x\Gamma\neq y\Gamma} \frac{|f(x\Gamma)-f(y\Gamma)|}{d(x\Gamma, y\Gamma)}.$$
\end{proposition}
\begin{proof} The right hand side of the above equation is equal to
$\sup_{x\neq y}  \frac{|f(x)-f(y)|}{d_G(x, y)}$. So it suffices to show
\begin{eqnarray} \label{Lipschitz:eq}
L_{\rho_0}(f)=\sup_{x\neq y} \frac{|f(x)-f(y)|}{d_G(x, y)}.
\end{eqnarray}

The proof is similar to that of \cite[Proposition 8.6]{Rieffel00}.
Let $\varrho: [0, 1]\rightarrow G$ be a continuously differentiable curve. Denote by $\ell(\varrho)$ the length of $\varrho$.
Then $(f\circ \varrho)'(\nu)=(\alpha_{-\Ad_{\varrho_{\nu}}(\varrho'_{\nu})}f)(\varrho_{\nu})$  for all $\nu \in [0, 1]$, and hence
\begin{eqnarray*}
|f(\varrho_1)-f(\varrho_0)|
&=&|\int^1_0(f\circ \varrho)'(\nu)\, d\nu |
\le  \int^1_0|(f\circ \varrho)'(\nu)| \, d\nu \\
&=& \int^1_0|(\alpha_{-\Ad_{\varrho_{\nu}}(\varrho'_{\nu})}f)(\varrho_{\nu})|\, d\nu
\le   \int^1_0\pa \alpha_{\Ad_{\varrho_{\nu}}(\varrho'_{\nu})}f\pa\, d\nu  \\
&\le & L_{\rho_0}(f) \int^1_0\pa \Ad_{\varrho_{\nu}}(\varrho'_{\nu})\pa\, d\nu
= L_{\rho_0}(f) \ell(\varrho),
\end{eqnarray*}
where in the last equality we use the fact that the Riemannian metric on $G$ is right translation invariant.
It follows easily that $|f(\varrho_1)-f(\varrho_0)|\le L_{\rho_0}(f) \ell(\varrho)$ holds if $\varrho$ is only piecewise continuously
differentiable. Considering all piecewise continuously
differentiable curves connecting $x$ and $y$ we obtain $|f(x)-f(y)|\le L_{\rho_0}(f)d_G(x, y)$ for all $x, y\in G$.

Denote by $e_G$ the identity element of $G$.
For any $0\neq X\in \fg$, we have
\begin{eqnarray*}
\sup_{x\neq y}\frac{|f(x)-f(y)|}{d_G(x, y)}
&\ge &\sup_{x}\sup_{\nu\neq 0}\frac{|f(x)-f(e^{\nu X}x)|}{d_G(x, e^{\nu X}x)} \\
&=&\sup_{\nu\neq 0}\sup_{x}\frac{|f(x)-f(e^{\nu X}x)|}{d_G(e_G, e^{\nu X})} \\
&\ge &\sup_{\nu\neq 0}\sup_{x}\frac{|f(x)-f(e^{\nu X}x)|}{|\nu| \pa X\pa} \\
&=& \sup_{\nu\neq 0}\frac{\pa f- \alpha_{e^{-\nu X}}f\pa}{|\nu| \pa X\pa}\\
&=& \sup_{\nu\neq 0}\frac{\pa \alpha_{e^{\nu X}}f-f\pa}{|\nu| \pa X\pa}
\ge  \frac{\pa \alpha_X(f)\pa}{\pa X\pa}.
\end{eqnarray*}
Therefore $\sup_{x\neq y}\frac{|f(x)-f(y)|}{d_G(x, y)}\ge L_{\rho_0}(f)$. This proves (\ref{Lipschitz:eq}).
\end{proof}

\section{Lip-norms and compact group actions} \label{action:sec}

In this section we recall the definition of compact quantum metric spaces and prove Theorem~\ref{action:thm},
which enables one to show that certain seminorm defines a quantum metric, via the help of a compact group action.

Rieffel has set up the theory of compact quantum metric spaces in the general framework of
order-unit spaces \cite[Defintion 2.1]{Rieffel00}.
We shall need it only for $C^*$-algebras.
By a \emph{$C^*$-algebraic compact quantum metric space} we mean a
pair $(A, L)$ consisting of a unital $C^*$-algebra
$A$ and  a (possibly $+\infty$-valued) seminorm $L$ on
$A$ satisfying
 the \emph{reality condition}
\begin{eqnarray} \label{real:eq}
L(a)&=& L(a^*)
\end{eqnarray}
for all $a\in A$, such that $L$ vanishes exactly on
$\Cb$ and the metric $d_L$ on the state space $S(A)$
defined by
\begin{eqnarray} \label{Lip to dist:eq}
d_L(\psi, \phi)=\sup_{L(a)\le 1} |\psi(a)-\phi(a)|
\end{eqnarray}
induces the weak${}^*$-topology.
The {\it radius}  of $(A, L)$, denote by $r_A$,  is defined to be the radius of $(S(A), d_L)$.
We say that $L$ is a {\it Lip-norm}.

Let $A$ be a unital $C^*$-algebra and let $L$ be a
(possibly $+\infty$-valued) seminorm on $A$ vanishing on
$\Cb$. Then $L$ and $\pa \cdot \pa$ induce (semi)norms ${\tilde
L}$ and $\pa \cdot\pa^{\sim}$ respectively on the quotient space
$\tilde{A}=A/\Cb$.

%
%
%
%
%

Recall that a {\it character} of a compact group is the trace function of a finite-dimensional complex representation
of the group \cite[Section II.4]{BD}.

\begin{lemma} \label{character:lemma}
Let $H$ be a compact group and $H_0$ be a closed normal subgroup of $H$ of finite index. Then for
any linear combination of finitely many characters of $H$, its multiplication with the characteristic function
of $H_0$ is also a linear combination of finitely many characters of $H$.
\end{lemma}
\begin{proof} The products and sums of characters of $H$ are still characters \cite[Proposition II.4.10]{BD}.
Thus it suffices to show that the characteristic function of $H_0$ on $H$ is a linear combination of finitely many characters of $H$.

Since $H/H_0$ is finite, every $\Cb$-valued class function on $H/H_0$, i.e., functions being constant
on conjugate classes, is a linear combination of characters of $H/H_0$ \cite[Proposition 2.30]{FH}.
Thus the characteristic function of $\{e_{H/H_0}\}$ on $H/H_0$, where $e_{H/H_0}$ denotes the identity
element of $H/H_0$,  is  a linear combination of characters
of $H/H_0$. Then the characteristic function $H_0$ on $H$ is a linear combination
of characters of $H$.
\end{proof}

Recall that a {\it length function} on a topological group $H$ is
a continuous $\Rb_{\ge 0}$-valued function, $\ell$, on $H$ such that
$\ell(h)=0$ if and only if $h$ is equal to the identity element
$e_H$ of $H$, that $\ell(h_1h_2)\le \ell(h_1)+\ell(h_2)$ for all $h_1, h_2\in H$,
 and that $\ell(h^{-1})=\ell(h)$ for all $h\in H$.

Suppose that a compact group $H$ has a strongly continuous action
$\sigma$ on a Banach space $A$ as isometric automorphisms. Endow $H$ with its normalized Haar measure.
For any continuous $\Cb$-valued function $\varphi$ on $H$, define a linear map
$\sigma_{\varphi}: A\rightarrow A$ by
$$ \sigma_{\varphi}(a)=\int_H\varphi(h)\sigma_h(a)\, dh$$
for $a\in A$. Denote by $\hat{H}$ the set of isomorphism  classes of irreducible representations of $H$.
For each $s\in \hat{H}$, denote by $A_s$ the spectral subspace of
$A$ corresponding to $s$. For a finite subset $J$ of $\hat{H}$, set $A_J=\sum_{s\in J}A_s$.

The main tool we use for the proof of Theorem~\ref{Lip:thm} will be the following
slight generalization of \cite[Theorem 4.1]{Li05}.

\begin{theorem} \label{action:thm}
Let $A$ be a unital $C^*$-algebra, let
$L$ be a (possibly $+\infty$-valued) seminorm on $A$
satisfying the reality condition (\ref{real:eq}),
and let $\sigma$ be a strongly continuous action of a compact group $H$
on $A$ by automorphisms. Assume that $L$ takes finite values on a dense
subspace of $A$, and that $L$ vanishes on $\Cb$.
Suppose that the following conditions are
satisfied:

\begin{enumerate}
\item there are some length function $\ell$ on a closed normal subgroup
$H_0$ of $H$ of finite index and some constant $C>0$ such that $L^{\ell}\le C\cdot L$ on
    $A$, where $L^{\ell}$ is the (possibly $+\infty$-valued)
seminorm on $A$ defined by
\begin{eqnarray} \label{def of L:eq}
L^{\ell}(a)=\sup \{\frac{\pa \sigma_h(a)-a\pa}{\ell(h)}| h\in H_0, h\neq e_H\}.
\end{eqnarray}

\item for any linear combination $\varphi$ of finitely many characters
    on $H$ we have $L\circ \sigma_{\varphi}\le \pa \varphi\pa_1\cdot
    L$ on $A$, where $\pa \varphi\pa_1$ denotes the $L^1$ norm of $\varphi$;

\item for each $s\in \hat{H}$ not being the trivial representation $s_0$ of $H$,  the set
    $\{a\in
    A_s| L(a)\le 1, \pa a\pa \le r\}$ is totally
    bounded for some $r>0$, and the only element in $A_s$
    vanishing under $L$ is $0$;

\item there is a unital $C^*$-algebra $\cA$  containing
the fixed-point subalgebra $A^{\sigma}$, with
a Lip-norm $L_{\cA}$, such
  that $L_{\cA}$ extends the restriction of $L$ to $A^{\sigma}$;

\item for each $s\in \widehat{H/H_0}\subseteq \hat{H}$ not equal to $s_0$,
there exists some constant $C_s>0$
such that $\pa \cdot \pa \le C_sL$ on $A_s$.
\end{enumerate}
Then $(A, L)$ is a $C^*$-algebraic compact quantum
metric space with $r_{A}\le C\int_{H_0}\ell(h)\, dh+\sum_{s_0\neq s\in \widehat{H/H_0}}C_s(\dim(s))^2+r_{\cA}$,
where $H_0$ is endowed with its normalized Haar measure.
\end{theorem}

We need some preparation for the proof of Theorem~\ref{action:thm}.
The following lemma generalizes \cite[Lemma 3.4]{Li05}.

\begin{lemma} \label{tensor limit vanish:lemma}
Let $H$ be a compact group, and let $H_0$ be a closed normal subgroup of $H$ of finite index.
Let $f$ be a continuous $\Cb$-valued function on $H$ with $f(e_H)=0$.
Then for any $\varepsilon>0$
there is a nonnegative function $\varphi$ on $H$ with support contained in $H_0$ such that
$\varphi$ is a linear combination of finitely many characters of $H$,
$\pa \varphi\pa_1=1$, and $\pa \varphi\cdot f\pa_1< \epsilon$.
\end{lemma}
\begin{proof}
Denote by $\chi$ the characteristic function of $H_0$ on $H$.
Set $g=f\chi+\varepsilon(1-\chi)$. Then $g\in C(H)$ and $g(e_H)=0$.
By \cite[Lemma 3.4]{Li05} we can find a nonnegative function $\phi$ on $H$ such
that $\phi$ is a linear combination of finitely many characters, $\pa \phi\pa_1=1$, and
$\pa \phi \cdot g\pa_1< \varepsilon/2$.  Then $\varepsilon\int_{H\setminus H_0}\phi(h)\, dh\le \pa \phi \cdot g\pa_1< \varepsilon/2$,
and hence
$$\pa \chi\phi\pa_1=\pa \phi\pa_1-\int_{H\setminus H_0}\phi(h)\, dh> 1-1/2= 1/2.$$
Set $\varphi=\chi\phi/\pa \chi\phi\pa_1$. By Lemma~\ref{character:lemma} $\varphi$ is a linear combination
of finitely many characters of $H$. One has
$$ \pa \varphi\cdot f\pa_1=\pa \chi\phi f\pa_1/\pa \chi\phi\pa_1=\pa \chi\phi g\pa_1/\pa \chi\phi\pa_1<(\varepsilon/2)/(1/2)=\varepsilon.$$
\end{proof}

For a compact group $H$ and  a finite subset $J$ of $\hat{H}$, set $\bar{J}=\{\bar{s}|s\in J\}$, where $\bar{s}$ denotes the
contragradient representation.
Replacing \cite[Lemma 3.4]{Li05} by Lemma~\ref{tensor limit vanish:lemma} in the proof
of \cite[Lemma 4.4]{Li05}, we get:

\begin{lemma} \label{finite approx:lemma}
Let $H$ be a compact group.
For any $\varepsilon>0$ there is a finite subset
$J=\bar{J}$ in
$\hat{H}$, containing the trivial representation $s_0$, depending only
on $\ell$ and $\varepsilon/C$, such that for any strongly continuous
isometric action $\sigma$ of $H$ on a complex Banach space $A$ with a
(possibly $+\infty$-valued) seminorm $L$ on $A$ satisfying conditions (1) and (2)
in Theorem~\ref{action:thm},
and any $a\in A$, there is some $a'\in A_{J}$ with
\begin{eqnarray*}
\pa a'\pa\le \pa a\pa, \quad L(a')\le
L(a), \quad \mbox{ and }\pa a-a'\pa\le \varepsilon L(a).
\end{eqnarray*}
If $A$ has an isometric involution being invariant under $\sigma$,
then when $a$ is self-adjoint
we can choose $a'$ also to be self-adjoint.
\end{lemma}

We are ready to prove Theorem~\ref{action:thm}.

\begin{proof}[Proof of Theorem~\ref{action:thm}]
Most part of the proof of \cite[Theorem 4.1]{Li05} carries over here. In fact, conditions (2)-(4) here
are the same as the conditions (2)-(4) in \cite[Theorem 4.1]{Li05}. Since the proof of Lemma 4.5
in \cite{Li05} does not involve condition (1) there, this lemma still holds in our current situation.
Replacing \cite[Lemma 4.4]{Li05} by Lemma~\ref{finite approx:lemma} in the proof of Lemma 4.6 of \cite{Li05},
we see that the latter also holds in our current situation.
To finish the proof of Theorem~\ref{action:thm}, we only need to prove the following
analogue of Lemma 4.7 of \cite{Li05}:

\begin{lemma} \label{radius:lemma}
We have
\begin{eqnarray*}
\pa \cdot \pa^{\sim}\le
(C\int_{H_0}\ell(h)\, dh+\sum_{s_0\neq s\in \widehat{H/H_0}}C_s(\dim(s))^2+r_{\cA})L^{\sim}
\end{eqnarray*}
on $(\tilde{A})_{\sa}$, where $H_0$ is endowed with its normalized Haar measure.
\end{lemma}
\begin{proof}
By Lemma~\ref{character:lemma} the characteristic function $\varphi$ of $H_0$ on $H$ is a linear combination
of characters of $H$.
Set $n=|H/H_0|$. Let $a\in A_{\sa}$. Then $\sigma_{n\varphi}(a)$ belongs to $A_{\sa}$ and is fixed by $\sigma|_{H_0}$.
We have
\begin{eqnarray*}
\pa a-\sigma_{n\varphi}(a)\pa
&=&\pa \int_{H_0}a\, dh-\int_{H_0}\sigma_h(a)\, dh\pa
\le \int_{H_0}\pa a-\sigma_h(a)\pa\, dh \\
&\le &L^{\ell}(a)\int_{H_0}\ell(h)\, dh
\le C\cdot L(a)\int_{H_0}\ell(h)\, dh,
\end{eqnarray*}
where the last inequality comes from the condition (1).
By the condition (2) we have
\begin{eqnarray*}
L(\sigma_{n\varphi}(a))\le \pa n\varphi\pa_1\cdot L(a)= L(a).
\end{eqnarray*}
Note that $A^{\sigma|_{H_0}}=\oplus_{s\in \widehat{H/H_0}}A_s$.
Say, $\sigma_{n\varphi}(a)=\sum_{s\in \widehat{H/H_0}}a_s$ with $a_s\in A_s$.
For each $s\in \widehat{H/H_0}$, denote by $\chi_s$ the corresponding character
of $H/H_0$, thought of as a character of $H$.
Then $a_s=\sigma_{\dim(s)\overline{\chi_s}}(\sigma_{n\varphi}(a))$ \cite[Lemma 3.2]{Li05}.
Thus
$$ L(a_s)=L(\sigma_{\dim(s)\overline{\chi_s}}(\sigma_{n\varphi}(a))\le \pa \dim(s)\overline{\chi_s}\pa_1L(\sigma_{n\varphi}(a))\le (\dim(s))^2L(a),$$
where the first inequality comes from the condition (2).
Note that $a_{s_0}\in A_{\sa}$. By the condition (5) we have
$$ \pa a_s\pa \le C_sL(a_s)\le C_s (\dim(s))^2L(a)$$
for each $s\in \widehat{H/H_0}$ not equal to $s_0$.
By the condition (4), we have
$$\pa b\pa^{\sim}\le r_{\cA} L^{\sim}(b)$$ for all $b\in (A_{s_0})_{\sa}=(A^{\sigma})_{\sa}$ \cite[Proposition 1.6, Theorem 1.9]{Rieffel98}
\cite[Proposition 2.11]{Li05}.  Thus
$$ \pa a_{s_0}\pa^{\sim}\le r_{\cA}L^{\sim}(a_{s_0})=r_{\cA}L(a_{s_0})\le r_{\cA}L(a).$$
Therefore we have
\begin{eqnarray*}
 \pa a\pa^{\sim}
&\le &\pa a-\sigma_{n\varphi}(a)\pa+\pa a_{s_0}\pa^{\sim}+ \sum_{s_0\neq s\in \widehat{H/H_0}}\pa a_s\pa \\
&\le & C\cdot L(a)\int_{H_0}\ell(h)\, dh+r_{\cA}L(a)+\sum_{s_0\neq s\in \widehat{H/H_0}}C_s(\dim(s))^2L(a)
\end{eqnarray*}
as desired.
\end{proof}
This finishes the proof of Theorem~\ref{action:thm}.
\end{proof}

\section{Proof of Theorem~\ref{Lip:thm}}  \label{Lip:sec}

In this section we prove Theorem~\ref{Lip:thm}.

Denote by $K_0$ the connected component of $K$ containing the identity element $e_K$.
 Take an inner product
on $\fk$ and use it to get a translation invariant
Riemannian metric on $K$ in the usual way. For each $x\in K_0$ set
$\ell(x)$ to be the geodesic distance form $e_K$ to $x$. Then
$\ell$ is a length function on $K_0$.

In order to prove Theorem~\ref{Lip:thm}, we just need to  verify the conditions
in Theorem~\ref{action:thm} for $(A, L, H, H_0, \sigma)=(C^*(\hat{G}/\Gamma, \rho), L_{\rho}, K, K_0, \beta)$.
Recall that we are given a norm on $\fg$, and
\begin{equation} \label{L:eqn}
L_{\rho}(a)=
\begin{cases}
\sup_{0\neq X\in \fg}\frac{\pa \alpha_X(a)\pa}{\pa X\pa}, &\text{if $a\in C^1(\hat{G}/\Gamma, \rho)$;}\\
\infty, &\text{otherwise,}
\end{cases}
\end{equation}
for $a\in C^*(\hat{G}/\Gamma, \rho)$.

By Lemma~\ref{cont:lemma} the actions $\alpha$ and $\beta$ on $C^*(\hat{G}/\Gamma, \rho)$ commute with each other.
Thus $\beta$ preserves $C^1(\hat{G}/\Gamma, \rho)$ and
$L_{\rho}$.


Choose the basis $X_1, \dots, X_{\dim(G)}$ of $\fg$ in
Proposition~\ref{der:prop} to be of norm $1$.
Denote by $C_1$
the supremum of $||F_{j, Y}||$ for all $1\le j\le \dim(G)$ and
$Y$ in the unit sphere of $\fk$ (with respect to
the inner product on $\fk$ above) in Proposition~\ref{der:prop}.

\begin{lemma} \label{C:lemma}
We have $L^{\ell}\le (dim(G)C_1)\cdot L_{\rho}$ on $C^*(\hat{G}/\Gamma, \rho)$.
\end{lemma}
\begin{proof}
It suffices to show $L^{\ell}\le (\dim(G)C_1)\cdot L_{\rho}$ on $C^1(\hat{G}/\Gamma, \rho)$.
By Proposition~\ref{der:prop} every $a\in C^1(\hat{G}/\Gamma, \rho)$ is
once differentiable with respect to the action $\beta$.
By \cite[Proposition 8.6]{Rieffel00} we have $L^{\ell}(a)=\sup_{Y\in \fk, ||Y||=1}||\beta_{Y}(a)||$.
Then from (\ref{der:eq}) in Proposition~\ref{der:prop} we get
$L^{\ell}(a)\le (\dim(G)C_1)L_{\rho}(a)$.
\end{proof}

\begin{lemma} \label{character1:lemma}
For any linear combination $\varphi$ of finitely many characters
    of $K$ we have $L_{\rho}\circ \beta_{\varphi}\le || \varphi ||_1\cdot
    L_{\rho}$ on $C^*(\hat{G}/\Gamma, \rho)$.
\end{lemma}
\begin{proof}
We have remarked above that $\beta$ preserves $L_{\rho}$.
By Lemma~\ref{derivation via length:lemma} one has
\begin{eqnarray} \label{L-length:eq}
L_{\rho}(a)=\sup_{0\neq X\in \fg}\frac{\pa \alpha_{e^{X}}(a)-a\pa}{\pa X\pa}
\end{eqnarray}
for every $a\in C^1(\hat{G}/\Gamma, \rho)$.
It follows that $L_{\rho}$ is lower semi-continuous on $C^1(\hat{G}/\Gamma, \rho)$
equipped with the relative topology from $C^1(\hat{G}/\Gamma, \rho)\subseteq C^*(\hat{G}/\Gamma, \rho)$.
By Lemma~\ref{C1:lemma} the action $\beta$ is also
strongly continuous on $C^1(\hat{G}/\Gamma, \rho)$ with respect to
the norm defined in Lemma~\ref{C1:lemma}.
Then $\beta_{\psi}$ is also well-defined on $C^1(\hat{G}/\Gamma, \rho)$
for any continuous $\Cb$-valued function $\psi$ on
$K$. By \cite[Remark 4.2.(3)]{Li05}
we get Lemma~\ref{character1:lemma}.
\end{proof}

The conditions (1) and (2) in Theorem~\ref{action:thm} for
$(A, L, H, H_0, \sigma)=(C^*(\hat{G}/\Gamma, \rho), L_{\rho}, K, K_0, \beta)$
follow from Lemmas~\ref{C:lemma} and \ref{character1:lemma} respectively.

Fix an inner product on $\fg$, and denote by $L'_{\rho}$ the seminorm on
$C^*(\hat{G}/\Gamma, \rho)$ defined by (\ref{L:eqn}) but using this inner product norm instead.
Since $\fg$ is finite dimensional, any two norms
on $\fg$ are equivalent. Therefore there exists some constant $C_2>0$ not depending on
$\rho$ such that $L'_{\rho}\le C_2L_{\rho}$.

By Lemma~\ref{quotient metric:lemma} and Proposition~\ref{Lipschitz:prop}
the restriction of $L'_{\rho_0}$ on $C^1(\hat{G}/\Gamma, \rho_0)\subseteq C(G/\Gamma)$ is the
Lipschitz seminorm associated to some metric $d$ on $G/\Gamma$.
The Arzela-Ascoli theorem \cite[Theorem VI.3.8]{Conway} tells us that
the set $\{a\in C^*(\hat{G}/\Gamma, \rho_0)| L_{\rho_0}(a)\le r_1, \pa a\pa \le r_2\}$ is
totally bounded for any $r_1, r_2>0$. Since for each $s\in \hat{K}$ neither the seminorm $L_{\rho}$ nor the $C^*$-norm
on $B_s\subseteq C^*(\hat{G}/\Gamma, \rho)$  depends on $\rho$, the condition (3) in Theorem~\ref{action:thm} for
$(A, L, H, H_0, \sigma)=(C^*(\hat{G}/\Gamma, \rho), L_{\rho}, K, K_0, \beta)$
follows.

From the criterion of Lip-norms in \cite[Proposition 1.6, Theorem 1.9]{Rieffel98} (see also
\cite[Proposition 2.11]{Li05}) one sees that the Lipschitz seminorm associated to
the metric on any compact metric space is a Lip-norm on the $C^*$-algebra of continuous
functions on this space. Since $L'_{\rho_0}$ on $C^*(\hat{G}/\Gamma, \rho_0)=C(G/\Gamma)$ is no less than
the Lipschitz seminorm associated to the metric $d$ on $G/\Gamma$, from \cite[Proposition 1.6, Theorem 1.9]{Rieffel98}
one concludes that $L_{\rho_0}$ is also a Lip-norm on $C(G/\Gamma)$. Therefore
we may take $(\cA, L_{\cA})$ in condition (4) of Theorem~\ref{action:thm} to
be $(C(G/\Gamma), L_{\rho_0})$ for $(A, L, H, H_0, \sigma)=(C^*(\hat{G}/\Gamma, \rho), L_{\rho}, K, K_0, \beta)$.

Let $s\in \hat{K}$ not being the trivial representation of $K$, and let
$a\in B_s$.
Then $L'_{\rho_0}(a)\le C_2L_{\rho_0}(a)=C_2L_{\rho}(a)$.
Thus for any $\lambda$ in the range of $a$ on $G/\Gamma$ one has
$\pa a-\lambda 1_{C(G/\Gamma)}\pa_{C(G/\Gamma)}\le C_2C_3L_{\rho}(a)$,
where $C_3$ denotes the diameter of $G/\Gamma$ under the metric $d$.  We have
\begin{eqnarray*}
\pa a\pa_{C^*(\hat{G}/\Gamma, \rho)}&=&\pa a\pa_{C(G/\Gamma)}
=\pa \int_K\overline{\left<k, s\right>}\beta_k(a-\lambda 1_{C(G/\Gamma)})\, dk\pa_{C(G/\Gamma)}\\
&\le &\pa a-\lambda 1_{C(G/\Gamma)}\pa_{C(G/\Gamma)}\le C_2C_3L_{\rho}(a).
\end{eqnarray*}
This establishes the condition (5) of Theorem~\ref{action:thm} for
$(A, L, H, H_0, \sigma)=(C^*(\hat{G}/\Gamma, \rho), L_{\rho}, K, K_0, \beta)$.

We have shown that the conditions in Theorem~\ref{action:thm} hold for $(A, L, H, H_0, \sigma)=(C^*(\hat{G}/\Gamma, \rho), L_{\rho}, K, K_0, \beta)$.
Thus Theorem~\ref{Lip:thm} follows from Theorem~\ref{action:thm}.
\section{Quantum Gromov-Hausdorff distance} \label{GH:sec}

In this section we prove Theorem~\ref{continuity:thm}.

We recall first the definition of the distance $\distnu$ from \cite[Section 5]{KL}. To simplify the notation,
for fixed unital $C^*$-algebras $A_1$ and $A_2$, when we take infimum over unital $C^*$-algebras $B$ containing
both $A_1$ and $A_2$, we mean to take infimum over all unital isometric $*$-homomorphisms of $A_1$ and $A_2$
into some unital $C^*$-algebra $B$. Denote by $\dist^B_{\rH}$ the Hausdorff distance between subsets of $B$.
For a $C^*$-algebraic compact quantum metric spaces $(A, L_{A})$, set
$$ \cE(A)=\{a\in A_{\sa}| L_A(a)\le 1\}.$$
 For any $C^*$-algebraic compact quantum metric spaces $(A_1, L_{A_1})$ and $(A_2, L_{A_2})$, the distance
$\distnu(A_1, A_2)$ is defined as
$$ \distnu(A_1, A_2)=\inf \dist^B_{\rH}(\cE(A_1), \cE(A_2)),$$
where the infimum is taken over all unital $C^*$-algebras $B$ containing $A_1$ and $A_2$.

Throughout the rest of this section, we fix $G$, $\Gamma$, $K$ such that there exits
$\rho$ satisfying the conditions (S1)-(S5). We also fix a norm on $\fg$.
Denote by $\Omega$ the set of all $\rho$ satisfying the conditions (S1) and (S2),
equipped with the  weakest topology  making the maps $\Omega\rightarrow G$ sending $\rho$ to $\rho(s)$
to be continuous for each $s\in \hat{K}$.

Every closed subgroup of a Lie group is also a Lie group \cite[Theorem 3.42]{Warner}.
Thus $K$ is a compact abelian Lie group. Then $K$ is the product of a torus and a finite
abelian group \cite[Corollary 3.7]{BD}. Therefore $\hat{K}$ is finitely generated.
Let $s_1, \dots, s_n$ be a finite subset of $\hat{K}$ generating $\hat{K}$. Then
the map $\varphi: \Omega\rightarrow \prod_{j=1}^nG$ sending $\rho$ to $(\rho(s_1), \dots, \rho(s_n))$
is injective, and its image is closed. Furthermore, it is easily checked that
the topology on $\Omega$ is exactly the pullback of the relative topology
of $\varphi(\Omega)$ in $\prod_{j=1}^nG$.
Since $G$ is a Lie group, it is locally compact metrizable. Thus $\prod^n_{j=1}G$ and $\Omega$ are
also locally compact metrizable.

For clarity and convenience, we shall denote the actions $\alpha$ and $\beta$ on $C^*(\hat{G}/\Gamma, \rho)$ by $\alpha_{\rho}$
and $\beta_{\rho}$ respectively, and denote the $C^*$-norm on $C^*(\hat{G}/\Gamma, \rho)$ by $\pa \cdot \pa_{\rho}$.
Consider the (possibly $+\infty$-valued) auxiliary seminorm $L''_{\rho}$ on
$C^*(\hat{G}/\Gamma, \rho)$ defined by
$$ L''_{\rho}(a)=\sup_{0\neq X\in \fg}\frac{\pa \alpha_{\rho, e^X}(a)-a\pa_{\rho}}{\pa X\pa}.$$

\begin{lemma} \label{lower semicontinuity:lemma}
Let $W$ be a locally compact Hausdorff space with a continuous map $W\rightarrow \Omega$ sending
$w$ to $\rho_w$.
Let $f$ be a continuous section of the continuous field of $C^*$-algebras over $W$ in Proposition~\ref{cont field:prop}.
Then the function $w\mapsto L''_{\rho_w}(f_w)$
is lower semi-continuous on $W$.
\end{lemma}
\begin{proof}
Let $w'\in W$. To show that the above function is lower semi-continuous at $w'$, we consider the case
$L''_{\rho_{w'}}(f_{w'})<\infty$.
The case $L''_{\rho_{w'}}(f_{w'})=\infty$
can be dealt with similarly.
Let $\varepsilon>0$. Take $0\neq X\in \fg$ such
that
$$L''_{\rho_{w'}}(f_{w'})\pa X\pa < \pa \alpha_{\rho_{w'}, e^X}(f_{w'})-f_{w'}\pa_{\rho_{w'}}+\varepsilon\pa X\pa.$$
It is easily checked that $w\mapsto \alpha_{\rho_w, e^X}(f_{w})$ is also a continuous section of the continuous field.
Then when $w$ is close enough to $w'$, we have
$$\pa \alpha_{\rho_{w'}, e^X}(f_{w'})-f_{w'}\pa_{\rho_{w'}}< \pa \alpha_{\rho_w, e^X}(f_{w})-f_{w}\pa_{\rho_{w}}+\varepsilon\pa X\pa$$
and hence
$$L''_{\rho_{w'}}(f_{w'})\pa X\pa < \pa \alpha_{\rho_w, e^X}(f_{w})-f_{w}\pa_{\rho_w}+2\varepsilon\pa X\pa
\le  (L''_{\rho_w}(f_{w})+2\varepsilon)\pa X\pa.$$
Therefore $L''_{\rho_{w'}}(f_{w'})\le L''_{\rho_w}(f_{w})+2\varepsilon$.
\end{proof}

Note that although the $*$-algebra structure of $C_{0, 1}(G/\Gamma, \rho)$ ($C_{\rb, 1}(G, \rho)$ resp.) depends on $\rho$, the Banach space structure, the left translation action of $G$ and the right translation action of $K$ on $C_{0, 1}(G/\Gamma, \rho)$ ($C_{\rb, 1}(G, \rho)$ resp.) do not depend
on $\rho$. Thus we may denote by $C_{0, 1}(G/\Gamma)$, $\alpha$ and $\beta$ this Banach space and these actions respectively.
Also denote by $C^1_{0, 1}(G/\Gamma)$ the set of once differentiable elements of $C_{0, 1}(G/\Gamma)$ with respect to $\alpha$.

\begin{lemma} \label{cont L:lemma}
For any $a$ in  $\oplus_{s\in \hat{K}}(B_s\cap C^1_{0, 1}(G/\Gamma))$, the function $\rho \mapsto L_{\rho}(a)$ is continuous on $\Omega$.
\end{lemma}
\begin{proof}
Say, $a=\sum_{s\in F}a_s$ for some finite subset $F$ of $\hat{K}$ and $a_s\in B_s\cap C^1_{0, 1}(G/\Gamma)$ for
each $s\in F$. Then $L_{\rho}(a)=\sup_{X\in \fg, \pa X\pa=1}\pa \sum_{s\in F}\alpha_X(a_s)\pa_{\rho}$ for each $\rho\in \Omega$.
Since $\alpha$ commutes with $\beta$, we have $\alpha_X(a_s)\in B_s$. By Proposition~\ref{cont field:prop}
the function $\rho\mapsto \pa \sum_{s\in F}\alpha_X(a_s)\pa_{\rho}$ is continuous on $\Omega$ for each $X\in \fg$.
Since $\fg$ is a finite-dimensional vector space and $\alpha_X(a_s)$ depends on $X$ linearly,
it follows easily that the function $(X, \rho)\mapsto \pa \sum_{s\in F}\alpha_X(a_s)\pa_{\rho}$ is continuous
on $\fg\times \Omega$. As the unit sphere of $\fg$ is compact, one concludes that the function
$\rho \mapsto \sup_{X\in \fg, \pa X\pa=1}\pa \sum_{s\in F}\alpha_X(a_s)\pa_{\rho}$ is continuous on $\Omega$.
\end{proof}

Fix $\rho'\in \Omega$. Let $Z$ be a compact neighborhood of $\rho'$ in $\Omega$.

Note that the linear span of $\rho\mapsto f(\rho)a\in C^*(\hat{G}/\Gamma, \rho)$
for $a$ in some $B_s$ and $f\in C(Z)$ is dense in
the $C^*$-algebra of continuous sections of the continuous field over $Z$ in Proposition~\ref{cont field:prop}.
Since $Z$ is a compact metrizable space, $C(Z)$ is separable.
As $G$ is a Lie group, it is separable. Then $G/\Gamma$ is separable, and hence is a compact metrizable space.
Thus $C(G/\Gamma)$ is separable, and hence $B_s$ is separable for each $s\in \hat{K}$. On the other hand,
since $\hat{K}$ is finitely generated, $\hat{K}$ is countable.
Therefore the $C^*$-algebra of continuous sections of the continuous field over $Z$ in Proposition~\ref{cont field:prop}
is separable.

By Proposition~\ref{nuclear:prop} each $C^*(\hat{G}/\Gamma, \rho)$ is nuclear.
Every separable continuous field of unital nuclear $C^*$-algebras over a compact metric space can be subtrivialized \cite[Theorem 3.2]{Blanchard}.
Thus we can find a unital $C^*$-algebra $B$ and unital embeddings $C^*(\hat{G}/\Gamma, \rho)\rightarrow B$ for
all $\rho \in Z$ such that, via identifying each $C^*(\hat{G}/\Gamma, \rho)$ with its image in $B$, the continuous sections of
the continuous field over $Z$ in Proposition~\ref{cont field:prop} are exactly the continuous maps $Z\rightarrow B$ whose images at each $\rho$ are in
$C^*(\hat{G}/\Gamma, \rho)$.

For any $C^*$-algebraic compact quantum metric space $(A, L_A)$ and any constant $R$ no less than the radius of $(A, L_A)$,
the set  $D_R(A):=\{a\in A_{\sa}|L_A(a)\le 1, \pa a\pa \le R\}$
is totally bounded and every $a\in \cE(A)$ can be written as $x+\lambda$ for some $x\in D_R(A)$ and $\lambda \in \Rb$
\cite[Proposition 1.6, Theorem 1.9]{Rieffel98}.
In Section~\ref{Lip:sec} we have seen that the conditions in Theorem~\ref{action:thm}
hold for $(A, L, H, H_0, \sigma)=(C^*(\hat{G}/\Gamma, \rho), L_{\rho}, K, K_0, \beta)$
with some $C, C_s$ and $(\cA, L_{\cA})$ not depending on $\rho$. Thus, by
Theorem~\ref{action:thm} there is some constant $R$ such that
the radius of $(C^*(\hat{G}/\Gamma, \rho_{\rho}), L_{\rho})$ is no bigger than $R$ for all $\rho \in \Omega$.
For any $\varepsilon>0$, by Lemmas~\ref{finite approx:lemma} and \ref{cont:lemma} there is a finite subset $F\subseteq \hat{K}$ satisfying that for
any $\rho \in \Omega$ and any $x\in \cE(C^*(\hat{G}/\Gamma, \rho))$ there is some $y\in \cE(C^*(\hat{G}/\Gamma, \rho))\cap \sum_{s\in F}B_s$
with $\pa y\pa_{\rho}\le \pa x\pa_{\rho}$ and $\pa x-y\pa_{\rho}<\varepsilon$.

\begin{lemma} \label{convergence1:lemma}
Let  $\varepsilon>0$. Then there is a neighborhood $U$ of $\rho'$ in $Z$ such that
for any $\rho \in U$ and any $a\in \cE(C^*(\hat{G}/\Gamma, \rho'))$ there is some
$b\in \cE(C^*(\hat{G}/\Gamma, \rho))$ with $\pa a-b\pa_B<\varepsilon$.
\end{lemma}
\begin{proof}
According to the discussion above we can find a finite subset $Y$ of $\cE(C^*(\hat{G}/\Gamma, \rho'))\cap \sum_{s\in F}B_s$ such that
for every $a\in  \cE(C^*(\hat{G}/\Gamma, \rho'))$ there are some $z\in Y$ and $\lambda \in \Rb$
with $\pa a-(z+\lambda)\pa_{\rho'}<\varepsilon$. For each $y\in Y$, write $y$ as $\sum_{s\in F}y_s$ with $y_s\in B_s$.
Since $L_{\rho'}(y)<\infty$, $y$ is once differentiable with respect to $\alpha_{\rho'}$.
It is easy to see that each $y_s$ is once differentiable with respect to $\alpha_{\rho'}$.
Thus, by Lemma~\ref{cont L:lemma} the function $\rho\mapsto L_{\rho}(y)$ is continuous on $\Omega$.
Then we can find a constant $\delta>0$  and a neighborhood $U$ of $\rho'$ in $Z$
such that $\delta\pa y_{\rho}\pa_{\rho}<\varepsilon$, $\pa y_{\rho'}-y_{\rho}\pa_B<\varepsilon$, and $L_{\rho}(y_{\rho})< 1+\delta$ for all $y\in Y$ and $\rho\in U$,
where $y_{\rho}$ denotes $y$ as an element in $C^*(\hat{G}/\Gamma, \rho)$.
Fix $\rho\in U$. Set $b=z_{\rho}/(1+\delta)$. Then $L_{\rho}(b+\lambda)=L_{\rho}(b)<1$, and
\begin{eqnarray*}
\pa a-(b+\lambda)\pa_B&\le &\pa a-(z_{\rho'}+\lambda)\pa_{\rho'}+\pa z_{\rho'}-z_{\rho}\pa_B+\pa z_{\rho}-b\pa_{\rho}\\
&<& \varepsilon+\varepsilon+ \varepsilon=3\varepsilon.
\end{eqnarray*}
\end{proof}

\begin{lemma} \label{convergence2:lemma}
Let $\varepsilon>0$. Then there is a neighborhood $U$ of $\rho'$ in $Z$ such that
for any $\rho \in U$ and any $a\in \cE(C^*(\hat{G}/\Gamma, \rho))$ there is some
$b\in \cE(C^*(\hat{G}/\Gamma, \rho'))$ with $\pa a-b\pa_B<\varepsilon$.
\end{lemma}
\begin{proof} According to the discussion before Lemma~\ref{convergence1:lemma}, it suffices to
show that
there is a neighborhood $U$ of $\rho'$ in $Z$ such that
for any $\rho\in U$ and any $a\in \cE(C^*(\hat{G}/\Gamma, \rho))\cap \oplus_{s\in F}B_s$ satisfying $\pa a\pa_{\rho} \le R$ there is some
$b\in \cE(C^*(\hat{G}/\Gamma, \rho'))$ with $\pa a-b\pa_B<\varepsilon$.
Suppose that this fails. Then we can find
a sequence $\{\rho_n\}_{n\in \Nb}$ in $Z$ converging
to $\rho'$ and an $a_n\in \cE(C^*(\hat{G}/\Gamma, \rho_n))\cap \oplus_{s\in F}B_s$ satisfying $\pa a_n\pa_{\rho_n} \le R$ for each $n\in \Nb$
such that $\pa a_n-b\pa_B\ge \varepsilon$ for all $n\in \Nb$ and $b\in  \cE(C^*(\hat{G}/\Gamma, \rho'))$.
Write $a_n$ as $\sum_{s\in F}a_{n, s}$ with $a_{n, s}\in B_s$. Then $a_{n, s}=\int_K\overline{\left<k, s\right>}\beta_{\rho_n, k}(a_n)\, dk$.
Thus $\pa a_{n, s}\pa_{\rho_n}\le \pa a_n\pa_{\rho_n}\le R$ and $L_{\rho_n}(a_{n, s})\le L_{\rho_n}(a_n)\le 1$
by Lemma~\ref{character1:lemma}.
Since the restriction of $L_{\rho}$ on $B_s$ does not depend on $\rho$, and the set
$\{a\in B_s| L_{\rho}(a)\le 1, \pa a\pa\le R\}$ is totally bounded,
passing to a subsequence
if necessary, we may assume that $a_{n, s}$ converges to some $a_s$ in $B_s$ when $n\to \infty$ for each $s\in F$.
Set $a=\sum_{s\in F}a_s$.  Then $(a_n)_{\rho_n}$ converges to $a_{\rho'}$ in $B$ as $n\to \infty$, where $(a_n)_{\rho_n}$ and $a_{\rho'}$ denote
$a_n$ and $a$ as elements in $C^*(\hat{G}/\Gamma, \rho_n)$ and $C^*(\hat{G}/\Gamma, \rho')$ respectively.
In particular, $a$ is self-adjoint and $\pa a\pa_{\rho'} \le \lim_{n\to \infty}\pa a_n\pa_{\rho_n}\le R$.

By Lemma~\ref{derivation via length:lemma} we have $L''_{\rho_n}(a_n)=L_{\rho_n}(a_n)\le 1$ for all $n\in \Nb$.
On the one-point compactification $W=\Nb\cup\{\infty\}$ of $\Nb$, consider the continuous map $W\rightarrow \Omega$
sending $n\in \Nb$ to $\rho_n$ and $\infty$ to $\rho'$.
Then the section $f$ defined as $f_n=a_n\in C^*(\hat{G}/\Gamma, \rho_n)$ for $n\in \Nb$
and $f_{\infty}=a\in C^*(\hat{G}/\Gamma, \rho')$  is a continuous section of the continuous field on $W$ in Proposition~\ref{cont field:prop}.      Thus, by Lemma~\ref{lower semicontinuity:lemma} we have $L''_{\rho'}(a)\le \liminf_{n\to \infty}L''_{\rho_n}(a_n)\le 1$.
By Lemma~\ref{approximate:lemma} we can find some self-adjoint $b\in C^1(\hat{G}/\Gamma, \rho')$ with
$\pa b\pa_{\rho'} \le \pa a\pa_{\rho'}\le R$, $\pa b-a\pa_{\rho'}\le \varepsilon/2$, and $L_{\rho'}(b)\le L''_{\rho'}(a)\le 1$.
Then $b\in \cE(C^*(\hat{G}/\Gamma, \rho'))$, and
$$\pa b-a_n\pa_B
\to \pa b-a\pa_{\rho'} \le \varepsilon/2$$
as $n\to \infty$.
Therefore, when $n$ is large enough, we have $\pa b-a_n\pa_B<\varepsilon$, contradicting
our assumption. This finishes the proof of the lemma.
\end{proof}

From Lemmas~\ref{convergence1:lemma} and \ref{convergence2:lemma} we conclude that Theorem~\ref{continuity:thm} holds.

\appendix
\section{Comparison of $\distnu$ and $\prox$} \label{comparison:appendix}

In this appendix we compare the distance $\distnu$ and the proximity Rieffel introduced
in \cite{Rieffel08}.

A (possibly $+\infty$-valued) seminorm $L$ on a unital (possibly incomplete) $C^*$-norm algebra $A$
is called a {\it $C^*$-metric} \cite[Definition 4.1]{Rieffel08} if
\begin{enumerate}
\item $L$ is lower semi-continuous,  satisfies the reality condition (\ref{real:eq}),
and is {\it strongly-Leibniz} in the sense that $L(ab)\le L(a)\pa b\pa +\pa a\pa L(b)$
for all $a, b\in A$, $L(1_A)=0$,  and $L(a^{-1})\le \pa a^{-1}\pa^2L(a)$ for all
$a$ being invertible in $A$,
\item $L$ extended to  the completion $\bar{A}$ of $A$ by $L(a)=\infty$ for $a\in \bar{A}\setminus A$
is a Lip-norm on $\bar{A}$,
\item  the algebra $\{a\in A| L(a)<\infty\}$ is spectrally stable in $\bar{A}$.
\end{enumerate}
In such case, the pair $(A, L)$ is called a {\it compact $C^*$-metric space}.

The seminorm $L_{\rho}$ in Theorem~\ref{Lip:thm} may fail to be a $C^*$-metric
since it may fail to be lower semi-continuous. However, it is lower semi-continuous
on $C^1(\hat{G}/\Gamma, \rho)$ by Lemma~\ref{derivation via length:lemma}. Thus its restriction on the algebra of
smooth elements in $C^*(\hat{G}/\Gamma, \rho)$ with respect to $\alpha$ is a $C^*$-metric.
By \cite[Proposition 3.2]{Rieffel08} its closure $\bar{L}_{\rho}$ is
a $C^*$-metric on $C^*(\hat{G}/\Gamma, \rho)$. Lemma~\ref{approximate:lemma} tells us
that
$$ \bar{L}_{\rho}(a)=\sup_{0\neq X\in \fg}\frac{\pa \alpha_{e^X}(a)-a\pa}{\pa X\pa}$$
for all $a\in C^*(\hat{G}/\Gamma, \rho)$.

In \cite[Definition 5.6, Section 14]{Rieffel08} Rieffel introduced the notions
of {\it proximity} $\prox(A, B)$ and {\it complete proximity}
$\prox_s(A, B)$ between two compact $C^*$-metric spaces $(A, L_A)$ and $(B, L_B)$.
In general, one has $\prox_s(A, B)\ge \prox(A, B)$. For each $q\in \Nb$, denote
by $\UCP_q(A)$ the set of unital completely positive linear maps from
the completion $\bar{A}$ of $A$ to $M_q(\Cb)$. Define
$\prox^q(A, B)$ as the infimum of the Hausdorff distance
of $\UCP_q(A)$ and $\UCP_q(B)$ in $\UCP_q(A\oplus B)$ under
the metric $d^q_L$,  for $L$ running through $C^*$-metrics $L$
on $A\oplus B$ whose quotients on $A$ and $B$ agree with
$L_A$ and $L_B$ on $A_{\sa}$ and $B_{\sa}$ respectively.
Here the metric $d^q_L$ is defined
as
$$ d^q_L(\varphi, \psi)=\sup_{L(a, b)\le 1}\pa \varphi(a, b)-\psi(a, b)\pa.$$
Then $\prox_s(A, B)$ is defined as $\sup_{q}\prox^q(A, B)$.

Note that
the definition of $\distnu$ extends to  compact $C^*$-metric spaces $(A, L_A)$ and $(B, L_B)$
directly.

\begin{theorem} \label{compare:thm}
For any compact $C^*$-metric spaces $(A, L_A)$ and $(B, L_B)$, one
has
$$ \distnu(A, B)\ge \prox_s(A, B).$$
\end{theorem}
\begin{proof} The proof is similar to those of \cite[Proposition 4.7]{Li06}
and \cite[Theorem 3.7]{KL}. Let $\cA$ be a unital $C^*$-algebra containing
$\bar{A}$ and $\bar{B}$. Set $c=\dist^{\cA}_{\rH}(\cE(A), \cE(B))$.
Let $\varepsilon>0$. Define a seminorm $L$ on $A\oplus B$ by
$$ L(a, b)=\max(L_A(a), L_B(b), \frac{\pa a-b\pa}{c+\varepsilon}).$$
It was pointed in the proof of \cite[Proposition 4.7]{Li06} that $L$ extended to
$\overline{A\oplus B}=\bar{A}\oplus \bar{B}$ as in the condition (2) of the definition of $C^*$-metrics above
is a Lip-norm, and that the quotients of $L$ on $A$ and $B$ agree with $L_A$ and $L_B$ on
$A_{\sa}$ and $B_{\sa}$ respectively. It is readily checked that $L$ satisfies the conditions (1) and (3)
in the definition of $C^*$-metrics. Thus $L$ is a $C^*$-metric on $A\oplus B$.
For any $q\in \Nb$ and $\varphi\in \UCP_q(A)$, by Arveson's extension theorem \cite[Theorem 1.6.1]{BO} extend $\varphi$ to
a $\phi$ in $\UCP_q(\cA)$. Set $\psi$ to be the restriction of $\phi$ on $\bar{B}$.
For any $(a, b)\in \cE(A\oplus B)$ one has
$$ \pa \varphi(a, b)-\psi(a, b)\pa =\pa \varphi(a)-\psi(b)\pa=\pa \phi(a-b)\pa\le \pa a-b\pa\le c+\varepsilon.$$
Thus $d^q_L(\varphi, \psi)\le c+\varepsilon$. Similarly, for any $\psi'\in \UCP_q(B)$, we can
find some $\varphi'\in \UCP_q(A)$ with $d^q_L(\varphi', \psi')\le c+\varepsilon$. Therefore
$\prox^q(A, B)\le c+\varepsilon$. It follows that $\prox^q(A, B)\le \distnu(A, B)$, and hence $\prox_s(A, B)\le \distnu(A, B)$
as desired.
\end{proof}

It was pointed out in Section 5 of \cite{KL} that one has continuity
of quantum tori and $\theta$-deformation,
convergence of matrix algebras to integral coadjoint orbits of compact connected semisimple
Lie groups,
and approximation of quantum tori by finite quantum tori with respect to $\distnu$. It follows
from Theorem~\ref{compare:thm} that we also have such continuity, convergence and approximation
with respect to $\prox_s$ and $\prox$. In particular, this yields a new proof for
\cite[Theorem 14.1]{Rieffel08}.


\end{document}